\documentclass[11pt]{article}
\usepackage[applemac]{inputenc}
\usepackage[T1]{fontenc}
\usepackage{lmodern}
\usepackage[a4paper]{geometry}
\usepackage[english]{babel}
\usepackage{graphicx}
\usepackage{onlyamsmath}
\usepackage{amsfonts}
\usepackage{amssymb}
\usepackage{ntheorem}

\newcommand{\e}{\epsilon}
\newcommand{\se}{\sqrt{\epsilon}}
\newcommand{\E}{\mathbb{E}}
\newcommand{\Pro}{\mathbb{P}}

\newcommand{\ga}{\gamma}
\newcommand{\1}{\textbf{1}}

\newcommand{\xt}[1]{\mathbf{#1}}

\newtheorem{thm}{Theorem}[section]
\newtheorem{prop}{Proposition}[section]
\newtheorem{cor}{Corollary}[section]
\newtheorem{lem}{Lemma}[section]
\theoremstyle{nonumberplain}
\theorembodyfont{\upshape}
\newtheorem{proof}{Proof}

\title{Radiative Transport Limit for the Random Schrödinger Equation with Long-Range Correlations}

\author{Christophe Gomez\thanks{Department of Mathematics, Stanford University, Building 380, Sloan Hall
Stanford, California 94305 USA (chgomez@math.stanford.edu). Tel: +(1)650-726-1625. Fax: +(1)650-725-4066.}}

\begin{document}

\maketitle

\begin{abstract} 

In this paper we study the asymptotic phase space energy distribution of solution of the Schrödinger equation with a time-dependent random potential. The random potential is assumed to be with slowly decaying correlations. We show that the Wigner transform of a solution of the random Schrödinger equation converges in probability to the solution of a radiative transfer equation. Moreover, we show that this radiative transfer equation with long-range coupling has a regularizing effect on its solutions. Finally, we give an approximation of this equation in term of a fractional Laplacian. The derivations of these results are based on an asymptotic analysis using perturbed-test-functions, martingale techniques, and probabilistic representations. 

\end{abstract}

\begin{flushleft}
\end{flushleft}

\begin{flushleft}
\end{flushleft}

\section*{Introduction.}

The Schrödinger equation with a time-dependent random potential has been studied for a long time and arises in many application \cite{blomgren, erdos, erdos2, erdos3, ho,spohn}, for instance in wave propagation in random media under the paraxial or parabolic approximation \cite{bal3, bal0, bal,bal4, bal2}. The present work has been motivated by data collections in wave propagation experiments showing the possibility to encounter propagation medium with slowly decaying autocorrelation functions \cite{dolan, sidi}. These observations have stimulated this field of research \cite{garnier, marty, marty2, solna}, but mainly in one-dimensional propagation media. These media are very convenient for mathematical studies, but not relevant in many applications. Random media with long-range correlations have been also considered in motion of particles. In this contexts, the authors have shown that the deviation of the trajectory of a particle from its mean trajectory converges to a fractional Brownian motion \cite{fannjiang, komorowski, komorowski2}.

Let us consider the random Schrödinger equation 
\[\begin{split}
i\partial_t \phi+\frac{1}{2}\Delta_{\xt{x}}\phi-\se V(t,\xt{x})\phi&=0, \quad t\geq 0\text{ and }\xt{x}\in \mathbb{R}^d,\\
\phi(0,\xt{x})=\phi_0(\xt{x}),
\end{split}\]
with a random potential $V(t,\xt{x})$, which is a spatially and temporally homogeneous mean-zero random field. Here, $t\geq0$ represents the temporal variable, $\xt{x}\in\mathbb{R}^d$ the spatial variable with $d\geq 1$, and $\e\ll1$ is a small parameter which represents the relative strength of the random fluctuations. A classical tool to study the phase space energy density of solutions of the random Schrödinger equation is the Wigner transform defined by
\[ W(t,\xt{x},\xt{k})=\frac{1}{(2\pi)^d}\int d\xt{y} e^{i \xt{k}\cdot\xt{y}}\phi\Big(t,\xt{x}-\frac{\xt{y}}{2}\Big)\overline{\phi\Big(t,\xt{x}+\frac{\xt{y}}{2}\Big)}.\]
We refer to \cite{gerard,lions} for the basic properties of the Wigner transfom. In our problem the size of the random variations are small, so we have to wait for long propagation distance and large time propagation to observe significant effects. Consequently, we consider the rescaled field 
\[\phi_\e(t,\xt{x})=\phi\Big(\frac{t}{\e},\frac{\xt{x}}{\e}\Big)\]  
which satisfies the scaled random Schrödinger equation 
\[\begin{split}
i\e\partial_t \phi_\e+\frac{\e^2}{2}\Delta_{\xt{x}}\phi_\e-\se V\Big(\frac{t}{\e},\frac{\xt{x}}{\e}\Big)\phi_\e&=0, \quad t\geq 0\text{ and }\xt{x}\in \mathbb{R}^d,\\
\phi_\e(0,\xt{x})=\phi_{0,\e}(\xt{x}).
\end{split}\]
In this regime, we consider the scaled Wigner transfom given by
\[ W_\e(t,\xt{x},\xt{k})=\frac{1}{(2\pi)^d}\int d\xt{y} e^{i \xt{k}\cdot\xt{y}}\phi_\e\Big(t,\xt{x}-\e\frac{\xt{y}}{2}\Big)\overline{\phi_\e\Big(t,\xt{x}+\e\frac{\xt{y}}{2}\Big)}.\]
The scaled Wigner transform is well suited to study functions oscillating on scales of order $\e^{-1}$, since the difference between $\phi_\e(t,\xt{x}-\e\xt{y}/2)$ and $\phi_\e(t,\xt{x}+\e\xt{y}/2)$ is of order $\mathcal{O}(1)$. In several paper \cite{bal0, bal, erdos, erdos2, erdos3, ho, lukkarinen, spohn} it has been shown that the expectation of the Wigner transform $\E[W_\e(t,\xt{x},\xt{k})]$ converges as $\e$ goes to $0$ to the solution $W$ of the radiative transport equation 
\[
\partial_t W +\xt{k}\cdot\nabla_\xt{x} W= \int d\xt{p}\sigma(\xt{p},\xt{k})(W(t,\xt{x},\xt{p})-W(t,\xt{x},\xt{k})),
\]
where the transfer coefficient $\sigma(\xt{p},\xt{k})$ depends on the power spectrum of the two-point correlation function of the random potential $V$. This result holds under mixing assumptions on the random potential $V$. Moreover, it has been shown that the limit $W$ is often self-averaging, that is, $W_\e$ converges in probability to the deterministic limit $W$ for the weak topology on $L^2(\mathbb{R}^{2d})$ \cite{bal3, bal,bal4}.

In \cite{bal2} the authors study the Schrödinger equation with a random potential with either rapidly or slowly decaying correlations. This paper constitutes a first step in the study of wave propagation in random media of dimension strictly greater than $1$ with long-range correlations. In \cite{bal2} the authors study the field $\phi$ itself and not its phase space energy density. In the rapidly decorrelating case the authors show that the field and the phase space energy density evolve at the same scale, which is of order $\e^{-1}$. As a result, in this case the scale $\e^{-1}$ is universal in the sense that it does not depend on the random potential $V$. However, in the slowly deccorelating case, the authors have observed macroscopic effects happening on the field $\phi$ at a shorter scale $\e^{-1/(2\kappa_0)}$, with $\kappa_0>1/2$. On this scale, the asymptotic field has a random phase modulation given by a fractional Brownian motion with Hurst index $\kappa_0$. Let us note that the parameter $\kappa_0$ depends on the statistic of the random potential $V$, and therefore the scale at which we can observe significant effects on the field $\phi$ is no longer universal. Moreover, as we will see in this paper, at the scale $\e^{-1/(2\kappa_0)}$ the phase space energy of the field $\phi_\e$ is not affected. In fact, the phase space energy distribution in both cases, rapidly and slowly decaying correlations, evolve on the scale $\e^{-1}$.   

In this paper, we investigate the propagation of the phase space energy density of the field $\phi_\e$ in a random medium with slowly decaying correlations. We show in Theorem \ref{thasymptotic} that the Wigner transform of $\phi_\e$ converges in probability, for the weak topology on $L^2(\mathbb{R}^{2d})$, to the unique solution of a radiative transfer equation similar to the one obtained in \cite{bal0} under rapidly decaying correlations, but with an important difference coming from the long-range correlation assumption. In contrast with the rapidly decorelating case \cite{bal0}, the scattering coefficient $\Sigma(\xt{k})=\int d\xt{p} \sigma(\xt{k},\xt{p})=+\infty$ is not defined anymore. The radiative transfer equation is still well defined because of the difference $W(t,\xt{x},\xt{p})-W(t,\xt{x},\xt{k})$ which balances the singularity given by the long-range correlation assumption. This result shows a qualitative and thorough difference between the two cases, rapidly and slowly deccorelating cases. First, in contrast with the rapidly decorrelating case, for which the phase and the phase space density evolve on the same scale $\e^{-1}$, the phase of $\phi$ and its phase space energy evolve at different scales in the slowly decorrelating case, which are respectively $\e^{-1/(2\kappa_0)}$ and $\e^{-1}$. Moreover, we show in Theorem \ref{regularity} that the long-range correlation assumption implies a regularizing effect of the radiative transfer equation. This regularizing effect is a consequence of $\Sigma(\xt{k})=+\infty$, and cannot be observed under rapidly decaying correlations. Finally, we give in Theorem \ref{fractional} an approximation of the radiative transfer equation in term of fractional Laplacian, which permits to exhibit from its solutions a damping coefficient obeying a power law with exponent lying in $(0,1)$. This approximation corresponds to the long space and time diffusion approximation of the radiative transfer equation. 

The organization of this paper is as follows: In Section \ref{section1}, we present the random Schröndinger equation that will be studied in this paper. Then, we present the construction of the random potential that will be used to model the random perturbations. Finally, we introduce the long-range correlation assumption used throughout this paper. In Section 2, we state the main result of this paper. We describe the asymptotic evolution in long range random media of the phase space energy density of the solution of the random Schrödinger equation. In Section \ref{regularizing}, we present the regularizing effect of the radiative transfer equation. 
In Section \ref{probrepsec}, we present the probabilistic representation of the radiative transfer equation.  
In Section \ref{fractionalsec} we present the approximation of the radiative transfer equation in term of fractional Laplacian. Finally, Sections \ref{proofthlimit}, \ref{proofthreg}, \ref{proofproprep}, and \ref{proofthfrac} are devoted to the proofs of Theorem \ref{thasymptotic}, Theorem \ref{regularity}, Proposition \ref{probrep} and Theorem \ref{fractional}.

\section*{Acknowledgment}

This work was supported by AFOSR FA9550-10-1-0194 Grant. I wish to thank Lenya Ryzhik for his suggestions. 

\section{The Random Schrödinger equation}\label{section1}

This section introduces precisely the random Schrödinger equation that we study in this paper. We consider the dimensionless form of the Schrödinger equation on $\mathbb{R}^d$ with a time-dependent random potential:
\begin{equation}\label{schrodingereq0}
i \partial_t \phi+\frac{1}{2}\Delta_{\xt{x}} \phi-\e^{\frac{1-\ga}{2}} V\Big(\frac{t}{\e^{\gamma}},\xt{x}\Big)\phi =0,\end{equation}
with $\ga\in (0,1)$. The case $\ga\in(0,1)$ permits to address a more convenient study of the phase space energy density of the field $\phi$. The case $\ga=0$, which will not be addressed in this paper, leads to much more difficult algebra. In \eqref{schrodingereq0} the strength of the random perturbation is small, so we consider the rescaled field 
\[\phi_\e(t,\xt{x})=\phi\Big(\frac{t}{\e},\frac{\xt{x}}{\e}\Big)\]
to observe significant effects after a sufficiently large propagation distance and propagation time. Therefore, the rescaled field $\phi_\e$ satisfies the scaled Schrödinger equation 
\begin{equation}\label{schrodingereq}
i\e \partial_t \phi_\e+\frac{\e^2}{2}\Delta_{\xt{x}} \phi_\e-\e^{\frac{1-\ga}{2}} V\Big(\frac{t}{\e^{1+\gamma}},\frac{\xt{x}}{\e}\Big)\phi_\e =0\quad \text{with}\quad \phi_\e(0,\xt{x})=\phi_{0,\e}(\xt{x}),
\end{equation}
with $\ga \in (0,1)$. Here $\Delta_{\xt{x}}$ is the Laplacian on $\mathbb{R}^d$ given by $\Delta=\sum_{j=1}^d \partial_{x_j}^2$. $(V(t,\xt{x}),\xt{x}\in \mathbb{R}^d, t\geq0)$ is the random potential, whose properties are described in the next section. Moreover, the initial datum $\phi_{0,\e}(\xt{x})=\phi_{0,\e}(\xt{x},\zeta)$ is a random function with respect to a probability space $(S,\mu(d\zeta))$, and independent to the random potential $V$. This randomness on the initial data is called mixture of states. This terminology comes from the quantum mechanics, and the reason for introducing this additional randomness will be explained more precisely in Section \ref{wignersection}.

\subsection{Random potential}\label{grfield}

In this section we present the construction of the random potential $V$ which is considered throughout this paper. It is also a short remainder about some properties of Gaussian random fields that we use in the proof of Theorem \ref{thasymptotic}. All the results exposed in this section can be shown using the standard properties of Gaussian random fields presented in \cite{adler, adlertaylor} for instance.

In this paper, the random perturbations are modeled using a stationary continuous random process in space and time, denoted by $(V(t,\xt{x}),  t\geq 0, x \in \mathbb{R}^d)$. We construct our potential in the Fourier space as follows. Let $\widehat{R}_0$ be a nonnegative function with support included in a compact subset of $\mathbb{R}^d$ containing $0$, such that
$\widehat{R}_0\in L^1(\mathbb{R}^d)$, $\widehat{R}_0(-\xt{p})=\widehat{R}_0(\xt{p})$, and $\widehat{R}_0$ has a singularity in $0$. Let us consider
\[\mathcal{H}=\left\{\varphi \text{ such that } \int_{\mathbb{R}^d}d\xt{p}\widehat{R}_0(\xt{p})\lvert \varphi(\xt{p})\rvert^2 <+\infty\right\},\] 
which is a Hilbert space equipped with the inner product
\[ \big<\varphi,\psi\big>_{\mathcal{H}}=\int d\xt{p}\,\,\widehat{R}_0(\xt{p})\varphi (\xt{p}) \overline{\psi(\xt{p})}\quad \forall (\varphi,\psi)\in \mathcal{H}^2.\]
Let us consider $(\widehat{V}(t,\cdot))_{t\geq0}$ be a stationary continuous zero-mean Gaussian field on $\mathcal{H}'$ with autocorrelation function given by
\[\E[\widehat{V}(t_1,d\xt{p}_1)\widehat{V}(t_2,d\xt{p}_2)]=(2\pi)^d R(t_1-t_2,\xt{p}_1)\delta(\xt{p}_1+\xt{p}_2)\]
and
\[\E[\widehat{V}(t_1,d\xt{p}_1)\overline{\widehat{V}(t_2,d\xt{p}_2)}]=(2\pi)^d R(t_1-t_2,\xt{p}_1)\delta(\xt{p}_1-\xt{p}_2),\]
and where $\mathcal{H}'$ is the dual space of $\mathcal{H}$. 
In other words, $\forall n\in\mathbb{N}^\ast$, $\forall (\varphi_1,\dots,\varphi_n) \in {\mathcal{H}}^n$ and $\forall (t_1\dots,t_n)\in[0,+\infty)^n$, 
\[\big(\big<\widehat{V}(t_1),\varphi_1\big>_{\mathcal{H}',\mathcal{H}}\,\,,\dots,\,\,\big<\widehat{V}(t_n),\varphi_n\big>_{\mathcal{H}',\mathcal{H}}\big)\] 
is a zero-mean Gaussian vector with covariance matrix given by: $\forall(j,l)\in\{1,\dots,n\}^2$
\[
\mathbb{E}\left[\big<\widehat{V}(t_j),\varphi_j\big>_{\mathcal{H}',\mathcal{H}}\big<\widehat{V}(t_l),\varphi_l\big>_{\mathcal{H}',\mathcal{H}}\right]=\int_{\mathbb{R^d}}d\xt{p}\,\,\varphi_j(\xt{p})\varphi_l(-\xt{p})R(t_1-t_2,\xt{p})
\]
and
\[
\mathbb{E}\left[\big<\widehat{V}(t_j),\varphi_j\big>_{\mathcal{H}',\mathcal{H}}\overline{\big<\widehat{V}(t_l),\varphi_l\big>_{\mathcal{H}',\mathcal{H}}}\right]=\int_{\mathbb{R^d}}d\xt{p}\,\,\varphi_j(\xt{p})\overline{\varphi_l(\xt{p})}R(t_1-t_2,\xt{p}).
\]
Here, the spatial power spectrum is given by
\begin{equation}\label{rtilde}
R(t,\xt{p})=e^{-\mathfrak{g}(\xt{p})\lvert t\rvert }\widehat{R}_0(\xt{p}),
\end{equation} 
where the nonnegative function $\mathfrak{g}$ is the spectral gap. We assume that the spectral gap is symmetric, that is $\mathfrak{g}(\xt{p})=\mathfrak{g}(-\xt{p})$. Particular assumptions involving the spectral gap $\mathfrak{g}$ will be introduced at the end of this section to ensure the long-range correlation property of the potential $V$.

According to the shape of the autocorrelation function \eqref{rtilde}, we have the following proposition.
\begin{prop} 
Let 
\begin{equation}\label{filtration}\mathcal{F}_t=\sigma(\widehat{V}(s,\cdot),s\leq t)\end{equation}
be the $\sigma$-algebra generated by $(\widehat{V}(s,\cdot), s\leq t)$. We have 
\begin{equation}\label{markovesp}
\mathbb{E}\big[ \widehat{V}(t+h,\cdot) \vert \mathcal{F}_t\big]=e^{-\mathfrak{g}(\xt{p})h}\widehat{V}(t,\cdot)\end{equation}
and $\forall (\varphi,\psi)\in \mathcal{H}^2$
\begin{equation}\label{markovvar}\begin{split}
\mathbb{E}\Big[ \big<\widehat{V}(t+h),\varphi\big>_{\mathcal{H}',\mathcal{H}} \big<\widehat{V}(t+h),\psi\big>_{\mathcal{H}',\mathcal{H}}&-  \mathbb{E}\big[  \big<\widehat{V}(t+h),\varphi\big>_{\mathcal{H}',\mathcal{H}} \vert \mathcal{F}_t\big]\mathbb{E}\big[  \big<\widehat{V}(t+h),\psi\big>_{\mathcal{H}',\mathcal{H}}\vert \mathcal{F}_t\big]  \Big\vert \mathcal{F}_t\Big]\\
&=\int d\xt{p}\,\, \varphi(\xt{p})\psi(\xt{-p})\widehat{R}_0(\xt{p})\left(1-e^{-2\mathfrak{g}(\xt{p})h} \right).
\end{split}\end{equation}
\end{prop}
These two properties will be used in the proof of Theorem \ref{thasymptotic}, which is based on the perturbed-test-function method.

Let us note that $\big<\widehat{V},\varphi\big>_{\mathcal{H}',\mathcal{H}}$ is a real-valued gaussian process once $\varphi\in\mathcal{H}$ satisfies $\overline{\varphi(\xt{p})}=\varphi(-\xt{p})$. According to this last remark, let us introduce the real random potential $V$ defined by
\[V(t,\xt{x})=\big<\widehat{V}(t,\cdot), e_{\xt{x}}\big>_{\mathcal{H}',\mathcal{H}}=\frac{1}{(2\pi)^d}\int_{\mathbb{R}^d} \widehat{V}(t,d\xt{p})e^{i\xt{p}\cdot \xt{x}},\]
where $e_{\xt{x}}\in\mathcal{H}$ is defined by $e_{\xt{x}}(\xt{p})=e^{i\xt{p}\cdot\xt{x}}/(2\pi)^d$. Consequently, the random potential $V$ is a stationary real-valued zero-mean Gaussian field with a covariance function given by: $ \forall (t_1,t_2)\in[0,+\infty)^2$ and $\forall(\xt{x}_1,\xt{x}_2)\in \mathbb{R}^{2d}$
\begin{equation}\label{covfunc}\begin{split} R(t_1-t_2,\xt{x}_1-\xt{x}_2)=\mathbb{E}\big[V(t_1,\xt{x}_1)V(t_2,\xt{x}_2)\big]&=
\frac{1}{(2\pi)^{d}}\int d\xt{p}R(t_1-t_2,\xt{p})e^{i\xt{p}\cdot(\xt{x}_1-\xt{x}_2)}\\
&=\frac{1}{(2\pi)^{d+1}}\int d\omega d\xt{p}\widehat{R}(\omega,\xt{p})e^{i\omega(t_1-t_2)}e^{i\xt{p}\cdot(\xt{x}_1-\xt{x}_2)},\end{split}\end{equation} 
where
\begin{equation}\label{spectraldens}
\widehat{R}(\omega,\xt{p})=\frac{2\mathfrak{g}(\xt{p})\widehat{R}_0(\xt{p})}{\omega^2+\mathfrak{g}^2(\xt{p})}.
\end{equation}

According to the previous construction and \cite[Theorem 2.2.1]{adlertaylor} the random potential $V$ is continuous and bounded with probability one on each compact subset $K$ of $\mathbb{R}\times\mathbb{R}^d$. This fact comes from the relation 
\[\mathbb{E}\big[\big(V(t,x) -V(s,y)\big)^2\big]^{1/2} \leq C \left(\int d\xt{p}\widehat{R}_0(\xt{p})\right)\big(\lvert t-s\rvert+\lvert \xt{x}-\xt{y}\rvert \big),\]
$\forall(t,s,\xt{x},\xt{y})\in[0,T]^2\times K^2$. Moreover, we have the following proposition. 
\begin{prop}
$\forall \mu>0$, $\eta>0$, and $\forall K$ compact subset of $\mathbb{R}^d$ 
\begin{equation}\label{cg1}
\lim_{\e\to 0} \mathbb{P}\Big(\e^\mu\sup_{\xt{x}\in K}\sup_{t\in[0,T]}\Big\lvert V\Big(\frac{t}{\e^{1+\gamma}},\frac{\xt{x}}{\e}\Big)\Big\rvert > \eta\Big)=0.\end{equation}
\end{prop}
According to \cite[Theorem 2.1.1]{adlertaylor}, one can show that the limit \eqref{cg1} holds exponentially fast as $\e \to 0$.

\subsection{Slowly decorrelating assumption} \label{SDCAsec}

In this paper we are interested in a random potential with long-range correlations. Let us introduce some additional assumptions on the spectral gap $\mathfrak{g}$ of the spatial power spectrum \eqref{rtilde}, in order to give slowly decaying correlations to the random potential $V$.

Let us note that $\forall t\geq 0$, the random field $V(t,\cdot)$ has spatial slowly decaying correlations. In fact, if we freeze the temporal variable, the autocorrelation function of the random potential $V(t,\cdot)$ is given by
\[R(t,\xt{x})=\E[V(t,\xt{x}+\xt{y})V(t,\xt{y})]=\int d\xt{p}  \widehat{R}_0(\xt{p})e^{-i\xt{x}\cdot\xt{p}}\] 
where $\widehat{R}_0(\xt{p})$ is assumed to have a singularity in $0$. Therefore, $R(t,\cdot)\not\in L^1(\mathbb{R}^d)$ and $(V(t))_{t\geq 0}$ models a family of random fields on $\mathbb{R}^d$ with spatial long-range correlations which evolves with respect to time. However, since  \eqref{schrodingereq0} is a time evolution problem, we have to take care of the evolution of the random perturbation $V$ with respect to the temporal variable. In fact, as we will see, if the family $(V(t))_{t\geq 0}$ has rapidly decaying correlations with respect to time the evolution problem \eqref{schrodingereq0} behaves like in the mixing case addressed in \cite{bal6}. 
In a time evolution problem with random perturbations with rapidly decaying correlations in time, even if at each fixed time the spatial correlations are slowly decaying, the resulting time evolution problem behaves as if it has mixing properties. Consequently, we have to introduce a long-range correlation assumption with respect to the temporal variable. Let us note that $\forall (s,\xt{x},\xt{y})\in \mathbb{R}_+\times\mathbb{R}^{2d}$
\begin{equation}\label{SDCAcor}\int_0^{+\infty} dt \big\lvert\E\big[V(t+s,\xt{x}+\xt{y})V(s,\xt{y})\big]\big\rvert=+\infty \Longleftrightarrow \int d\xt{p}  \frac{\widehat{R}_0(\xt{p})}{\mathfrak{g}(\xt{p})} =+\infty.  \end{equation}
Consequently, throughout this paper we say that the family $(V(t))_{t\geq 0}$ of random fields with spatial long-range correlations has slowly decaying correlations in time if  
\begin{equation}\label{SDCA}
\int d\xt{p}\frac{\widehat{R}_0(\xt{p})}{\mathfrak{g}(\xt{p})}=+\infty,\end{equation}
and rapidly decaying correlation in time otherwise. For technical reasons in the proof of Theorem \ref{thasymptotic} we assume that
\begin{equation}\label{longrange}
\int d\xt{p}\widehat{R}_0(\xt{p})\frac{\lvert\xt{p}\rvert}{\mathfrak{g}(\xt{p})}+\int d\xt{p}\widehat{R}_0(\xt{p})\frac{\lvert\xt{p}\rvert^2}{\mathfrak{g}^2(\xt{p})}+\int d\xt{p}\widehat{R}_0(\xt{p})\frac{\lvert\xt{p}\rvert^3}{\mathfrak{g}^3(\xt{p})}<+\infty.
\end{equation}
This assumption is satisfied if $\lvert \xt{p}\rvert =\mathcal{O}\big(\mathfrak{g}(\xt{p})\big)$ as $\xt{p} \to 0$ for instance since $\widehat{R}_0\in L^1(\mathbb{R}^d)$. For the sake of simplicity, we assume throughout this paper that 
\begin{equation}\label{SDCAreg}
\frac{\widehat{R}_0(\xt{p})}{\mathfrak{g}(\xt{p})}\sim \frac{\sigma}{\lvert \xt{p}\rvert^{d+\theta}}\quad \text{as }\xt{p}\to0, \text{ with }\theta\in(0,1),
\end{equation}
and $\sigma>0$. This last assumption is used in the proof of Theorem \ref{regularity}. For example, in \cite{bal2} the authors have supposed that
\begin{equation}\label{hyplenya}
\mathfrak{g}(\xt{p})=\nu \lvert \xt{p}\rvert ^{2\beta}\quad\text{and}\quad \widehat{R}_0(\xt{p})=\frac{a(\xt{p})}{\lvert\xt{p}\rvert^{d+2\alpha-2}}.
\end{equation}
Here, $a$ is a continuous function with compact support and $a(0)>0$, $\nu>0$, $\beta\in(0,1/2]$, $\alpha\in(1/2,1)$, and $\alpha+\beta>1$. These assumptions permit to model a random field $V(t,\xt{x})$ with spatial long-range correlations for each time $t\geq 0$ and with slowly decaying correlations in time.

\subsection{Wigner transform}\label{wignersection}

To study the asymptotic phase space  energy propagation of the solution $\phi_\e$ of the Schrödinger equation \eqref{schrodingereq}, let us consider the Wigner distribution of $\phi_\e$ averaged with respect to the randomness of the initial data:
\[W_\e (t,\xt{x},\xt{k})=\frac{1}{(2\pi)^d}\int_{\mathbb{R}^d\times S}d\xt{y}\mu(d\zeta)e^{i\xt{k}\cdot\xt{y}} \phi_\e \Big(t,\xt{x}-\e\frac{\xt{y}}{2},\zeta\Big)\overline{\phi_\e \Big(t,\xt{x}+\e\frac{\xt{y}}{2},\zeta\Big)}.\]
We refer to \cite{gerard, lions} for the basic properties of the Wigner distribution. The Wigner distribution satisfies the following evolution equation
\begin{equation}\label{wignereq}\begin{split}
\partial_t W_\e(t,\xt{x},\xt{k})&+\xt{k}\cdot\nabla_{\xt{x}}W_\e(t,\xt{x},\xt{k})=\\
&\frac{1}{\e^{\frac{1+\gamma}{2}}}\int_{\mathbb{R}^d}\frac{\widehat{V}\Big(\frac{t}{\e^{1+\gamma}},d\xt{p}\Big)}{(2\pi)^d i} e^{i\xt{p}\cdot\xt{x}/\e}\Big( W_\e\Big(t,\xt{x},\xt{k}-\frac{\xt{p}}{2}\Big)-W_\e\Big(t,\xt{x},\xt{k}+\frac{\xt{p}}{2}\Big)\Big),
\end{split}\end{equation}
with initial conditions $W_\e(0,\xt{x},\xt{k})=W_{0,\e}(\xt{x},\xt{k})$, where $W_{0,\e}$ is the Wigner transform of the initial data $\phi_{0,\e}$ of \eqref{schrodingereq}. Equation \eqref{wignereq} can be recast in the weak sense as follows: $\forall \lambda\in \mathcal{C}^{\infty}_0(\mathbb{R}^{2d})$,
\[
\big<W_\e(t),\lambda\big>_{L^2(\mathbb{R}^{2d})}-\big<W_\e(0),\lambda\big>_{L^2(\mathbb{R}^{2d})}=\int_0^t \big<W_\e, \xt{k}\cdot\nabla_{\xt{x}}\lambda+\frac{1}{\e^{\frac{1+\gamma}{2}}}\mathcal{L}_\e(s)\lambda\big>_{L^2(\mathbb{R}^{2d})}ds,
\]
where
\begin{equation}\label{lepsilon}
\mathcal{L}_\e\lambda(t,\xt{x},\xt{k})=\frac{1}{(2\pi)^d i}\int_{\mathbb{R}^d}\widehat{V}\Big(\frac{t}{\e^{1+\ga}},d\xt{p}\Big) e^{i\xt{p}\cdot\xt{x}/\e}\Big( \lambda\Big(\xt{x},\xt{k}-\frac{\xt{p}}{2}\Big)-\lambda\Big(\xt{x},\xt{k}+\frac{\xt{p}}{2}\Big)\Big).
\end{equation}
Here,  $\mathcal{C}^{\infty}_0(\mathbb{R}^{2d})$ stands for the set of smooth functions with compact supports. 
In what follows, we assume that $W_{0,\e}$ converges weakly in $L^2(\mathbb{R}^{2d})$ to a limit $W_0$, that is, 
\begin{equation}\label{weakcvinit}
\forall \lambda\in L^2(\mathbb{R}^{2d}),\quad \lim_\e\big<W_{0,\e},\lambda\big>_{L^2(\mathbb{R}^{2d})}= \big<W_{0},\lambda\big>_{L^2(\mathbb{R}^{2d})}.
\end{equation}
 Let us recall that by the Banach-Steinhaus Theorem $W_{0,\e}$ is bounded in $L^2(\mathbb{R}^{2d})$ uniformly in $\e$. The main reason to introduce the additional randomness through the initial data $\phi_{0,\e}$ and $(S,\mu(d\xt{q}))$, is to make possible the weak convergence \eqref{weakcvinit}. Let us give an example of such a situation. Let $\mu(\xt{q})$ be a nonnegative rapidly decreasing function such that $\|\mu\|_{L^1(\mathbb{R}^d)}=1$. Then, $(\mathbb{R}^d,\mu(d\xt{q}))$ is a probability space, where $\mu(d\xt{q})=\mu(\xt{q})d\xt{q}$. Let $\phi_{0,\e}(\xt{x},\xt{q})=\phi_0(\xt{x})e^{-i\xt{q}\cdot \xt{x}/\e}$, and $\phi_0$ be a sufficiently smooth function. In this case, we have $\|W_{0,\e}\|_{L^2(\mathbb{R}^{2d})}\leq \|\mu\|_{L^2(\mathbb{R}^d)}\|\phi_0\|_{L^4(\mathbb{R}^d)}$, and $W_{0,\e}$ converge weakly in $L^2(\mathbb{R}^{2d})$ to $W_0(\xt{x},\xt{k})=\frac{1}{(2\pi)^d}\widehat{\mu}(\xt{k})\lvert \phi_0(\xt{x})\rvert^2$. In this example, the initial data $\phi_{0,\e}(\xt{x},\xt{q})$ depends on the phase vector $\xt{q}$, which is distributed according to the probability distribution $\mu(d\xt{q})$. Then, the averaged Wigner transform consists in taking the expectation of the Wigner transform with respect to the probability distribution $\mu(d\xt{q})$ of the phase vector $\xt{q}$.

\section{Asymptotic Evolution in Long Range Random Media}\label{asymptsec}

The two following sections present the asymptotic behavior of the phase and the phase space energy density of the solution of \eqref{schrodingereq}. The results in these two sections permit to show the qualitative difference between the random effects induced on a wave propagating in long-range random media in time and in rapidly decorrelating random media in time $\eqref{SDCAcor}$.

\subsection{Phase Evolution}

In this section we present the asymptotic behavior of the phase of $\phi_\e$ solution of \eqref{schrodingereq}. Theorem \ref{thphase} stated below has been shown \cite{bal2} in the case $\ga=0$. Nevertheless, the proof of Theorem \ref{thphase} remains the same as the one of \cite[Theorem 1.2]{bal2}, but we state this result in the case $\ga\in(0,1)$ in order to provide a self-contained presentation. Under the long-range correlation assumption in time and medium parameters given by \eqref{hyplenya}, we have the following result.
\begin{thm}\label{thphase}
Let us assume that the autocorrelation function $R(t,\xt{x})$ of the random perturbation is given by  \eqref{rtilde}, \eqref{covfunc}, and \eqref{hyplenya}, and let
\[\kappa_0=\frac{\alpha+2\beta-1}{2\beta}\quad \text{and}\quad \kappa_\ga=\frac{\kappa_0}{1-\ga\Big(\frac{\alpha+\beta-1}{\beta}\Big)}\quad \text{for }\ga\in(0,1).\]
 Let us consider the process $\widehat{\zeta}_{\kappa_\ga,\e}(t,\xi)$ defined by 
\[  \widehat{\zeta}_{\kappa_\ga,\e}(t,\xi)=\frac{1}{\e^{d/(2\kappa_\ga)}}\widehat{\phi}_{\kappa_\ga,\e}\Big(t,\frac{\xi}{\e^{1/(2\kappa_\ga)}}\Big),\] 
with
\[\phi_{\kappa_\ga,\e}(t,\xt{x})=\phi\Big(\frac{t}{\e^{1/(2\kappa_\ga)}},\frac{\xt{x}}{\e^{1/(2\kappa_\ga)}}\Big)\]
and where $\phi$ satisfies 
\[\begin{split}
i \partial_t \phi+\frac{1}{2}\Delta_{\xt{x}} \phi-\e^{\frac{1-\ga}{2}} V\Big(\frac{t}{\e^{\gamma}},\xt{x}\Big)\phi& =0,\\
\phi(0,\xt{x})&=\phi_0(\xt{x}).
\end{split}\]
Then, for each $t\geq0$, $\xi\in\mathbb{R}^d$ fixed, and $\ga\in(0,1)$, $\widehat{\zeta}_{\kappa_\ga,\e}(t,\xi)$ converges in distribution to
\[\zeta_0(t,\xi)=\widehat{\phi}_0(\xt{x})\exp\Big(i\sqrt{D(\kappa_\ga,\xi)}B_{\kappa_\ga}(t)\Big),\]
where $(B_{\kappa_\ga}(t))_t$ is a standard fractional Brownian motion with,
\[D(\kappa_\ga,\xi)=D(\kappa_\ga)=\frac{a(0)\Omega_d}{(2\pi)^d\kappa_\ga(2\kappa_\ga-1)}\int_0^{+\infty}d\rho \frac{e^{-\nu\rho^{2\beta}}}{\rho^{2\alpha-1}}.\]
Here, $\Omega_d$ is the surface area of the unit sphere in $\mathbb{R}^d$, and $\xt{e}_1\in \mathbb{S}^{d-1}$.
\end{thm}

This result means that the phase of the field $\phi$ evolves on the scale $\e^{-1/(2\kappa_\ga)}$ which is smaller than the one on which evolves the phase space energy density as we will see below in Theorem \ref{thasymptotic}. The random perturbations induce a random phase modulation of the fields $\phi$ given by a fractional Brownian motion. This result shows an important difference with the case of media with rapidly decorrelating random perturbations in time \eqref{SDCAcor}, for which the phase and the phase space energy density evolve on the same scale $\e^{-1}$ \cite{bal2}.

Let us note that the scale of evolution of the random phase is shorter in the case $\ga>0$ than in the case $\ga=0$, that is $\e^{-1/(2\kappa_\ga)}\ll \e^{-1/(2\kappa_0)}\ll \e^{-1}$. This fact comes from the random perturbations which are stronger in the case $\ga>0$ than in the case $\ga=0$. These strong perturbations in time lead to the development of a random phase modulation earlier than in the case $\ga=0$. However, the scale of evolution of the phase space energy distribution is the same $\e^{-1}$ in both cases, $\ga>0$ and $\ga=0$.

\subsection{Limit Theorem}\label{limitth}

This section presents the main result of this paper. Theorem \ref{thasymptotic} describes the asymptotic behavior of the phase space energy distribution of $\phi_\e$ solution of the random Schrödinger equation \eqref{schrodingereq}. First, let us introduce some notations. Thanks to \eqref{weakcvinit}, let $r=\sup_{\e}\|W_{0,\e}\|_{L^2(\mathbb{R}^{2d})}<\infty$,
\[\mathcal{B}_{r}=\left\{\lambda \in L^2(\mathbb{R}^{2d}), \|\lambda\|_{L^2(\mathbb{R}^{2d})}\leq r\right\}\] 
be the closed ball with radius $r$, and $\{g_n, n\geq 1\}$ be a dense subset of $\mathcal{B}_{r}$. We equip  $\mathcal{B}_{r}$ with the distance $d_{\mathcal{B}_{r}}$ defined by
\[d_{\mathcal{B}_{r}}(\lambda, \mu)=\sum_{j=1}^{+\infty}\frac{1}{2^j}\left\lvert\big<\lambda-\mu,g_n\big>_{L^2(\mathbb{R}^{2d})}\right\rvert\]
$\forall (\lambda,\mu)\in(\mathcal{B}_{r})^2$, so that $(\mathcal{B}_{r} ,d_{\mathcal{B}_{r}})$ is a compact metric space. Therefore, $(W_\e)_\e$ is a family of process with values in $(\mathcal{B}_{r} ,d_{\mathcal{B}_{r}})$, since $\|W_\e (t)\|_{L^2(\mathbb{R}^{2d})}=\|W_{0,\e}\|_{L^2(\mathbb{R}^{2d})}$.

\begin{thm}\label{thasymptotic}
The family $(W_\e)_{\e\in(0,1)}$, solution of the transport equation \eqref{wignereq}, converges in probability on 
$\mathcal{C}([0,+\infty), (\mathcal{B}_{r} ,d_{\mathcal{B}_{r}}))$ as $\e \to 0$ to a limit denoted by $W$. More precisely, $\forall T>0$ and $\forall \eta>0$,
\[\lim_{\e\to 0}\Pro\left(\sup_{t\in[0,T]}d_{\mathcal{B}_r}(W_\e(t),W(t))>\eta\right)=0.\]
$W$ is the unique weak solution uniformly bounded in $L^2(\mathbb{R}^{2d})$ of the radiative transfer equation
\begin{equation}\label{radtranseq}
\partial_t W+\xt{k}\cdot \nabla_{\xt{x}} W=\mathcal{L}W,
\end{equation} 
with $W(0,\xt{x},\xt{k})=W_0(\xt{x},\xt{k})$. 
Here, $\mathcal{L}$ is defined by
\begin{equation}\label{formgene}
\mathcal{L}\varphi(\xt{k})=\int d\xt{p}\sigma(\xt{p}-\xt{k})\big(\varphi(\xt{p})-\varphi(\xt{k})\big),
\end{equation}
with $\varphi\in\mathcal{C}^\infty (\mathbb{R}^d)$ and where 
\begin{equation}\label{sigmadef}\sigma(\xt{p})=\frac{2\widehat{R}_0(\xt{p})}{(2\pi)^d\mathfrak{g}(\xt{p})}.\end{equation}
\end{thm}
Here, $\mathcal{C}^\infty _0(\mathbb{R}^d)$ stands for the infinite differentiable functions with compact support. Let us precise that $W$ is a weak solution of the transfer equation \eqref{radtranseq} means that 
\[\big<W(t),\lambda\big>_{L^2(\mathbb{R}^d)}-\big<W_0,\lambda\big>_{L^2(\mathbb{R}^d)}-\int_0^t \big<W(s),\xt{k}\cdot\nabla_\xt{x} \lambda +\mathcal{L}\lambda\big>_{L^2(\mathbb{R}^d)}=0,\] 
$\forall t\geq 0$ and $\forall \lambda\in\mathcal{C}^\infty_0(\mathbb{R}^d)$.

The topology gererated by the metric $d_{\mathcal{B}_r}$  is equivalent to the weak topology on $L^2(\mathbb{R}^{2d})$ restricted to $\mathcal{B}_r$. We cannot expect a convergence on $L^2(\mathbb{R}^{2d})$ equipped with the strong topology. In fact, the following conservation relation $\|W_\e(t)\|_{L^2(\mathbb{R}^{2d})}=\|W_{0,\e}\|_{L^2(\mathbb{R}^{2d})}$ cannot be satisfied by the limit $W$. Moreover, let us note that the Wigner distribution $W_\e$ is self-averaging as $\e$ goes to $0$, that is the limit $W$ is not random anymore. This self-averaging phenomenon of the Wigner distribution has already been observed in several studies \cite{bal3, bal,bal4} and is very useful in applications.

The proof of Theorem \ref{thasymptotic} is given in Section \ref{proofthlimit} and is based on an asymptotic analysis using perturbed-test-function and martingale techniques.

Equation \eqref{radtranseq} describes the asymptotic evolution of the phase space  energy distribution of the field $\phi_\e$ solution of the random Schrödinger equation \eqref{schrodingereq}. The nonlocal transfer operator $\mathcal{L}$ defined by \eqref{formgene} describes the energy diffusion caused by the random perturbations, and the transfer coefficient $\sigma(\xt{p}-\xt{k})$ describes the energy transfer between the modes $\xt{k}$ and $\xt{p}$.  

Let us note that the result of Theorem \ref{thasymptotic} does not depend on whether $\int d\xt{p} \sigma(\xt{p})$ is finite or not. In other words, the radiative transfer equation \eqref{radtranseq} is valid in the two case, slowly and rapidly decaying correlations in time \eqref{SDCAcor}. Consequently, the phase space density energy distribution evolves on the same scale $\e^{-1}$ in both case, rapidly and slowly decaying correlations in time. However, as we will see in Section \ref{regularizing}, the solutions of these equations in the two cases behave in different ways. As it has been discussed in Section \ref{SDCAsec}, in the case of rapidly decaying correlations in time \eqref{SDCAcor}, that is $\int d\xt{p} \sigma(\xt{p})<+\infty$, the radiative transfer equation \eqref{radtranseq} has the same properties as in the mixing case addressed in \cite{bal6}. In the case of slowly decaying correlations in time \eqref{SDCAcor}, that is $\int d\xt{p} \sigma(\xt{p})=+\infty$, we observe a regularizing effect of the solution of  \eqref{radtranseq} which cannot be observed in the case of rapidly decaying correlations in time. As we will see in Theorem \ref{regularity}, the unique solution $W$ of \eqref{radtranseq} is actually the unique classical solution of \eqref{radtranseq}.

\section{Regularizing Effects of the Radiative Transfer Equation \eqref{radtranseq}}\label{regularizing}

In this section we investigate the regularizing effect of the radiative transfer equation \eqref{radtranseq}. We show that the solution of \eqref{radtranseq} becomes instantaneously a smooth function, that is upon $t>0$, despite the nonsmoothness of its initial data. 

\begin{thm}\label{regularity}
Let $W_0\in L^2 (\mathbb{R}^{2d})$ and $W$ be the unique weak solution of the radiative transfer equation \eqref{radtranseq}. Then, under \eqref{SDCAreg} and $\forall t_0>0$, we have 
\[
W\in \mathcal{C}^0\Big( (0,+\infty), \bigcap_{k\geq 0} H^k (\mathbb{R}^{2d}) \Big)\cap L^\infty \Big([t_0,+\infty),\bigcap _{k\geq 0} H^k (\mathbb{R}^{2d})\Big),
\]
so that we also have 
\[W\in \mathcal{C}^0\big( (0,+\infty),\mathcal{C}^\infty(\mathbb{R}^{2d})\big)\cap L^\infty \big( [t_0,+\infty),\mathcal{C}^\infty(\mathbb{R}^{2d})\big).\] 
Moreover, $W$ is given by
\[W(t,\xt{x},\xt{k})=\frac{1}{(2\pi)^{2d}}\int d\xt{y}d\xt{q} e^{i(\xt{x}\cdot\xt{y}+\xt{k}\cdot\xt{q})}e^{\int_0^t du \Psi(\xt{q}+u\xt{y})}\widehat{W}_0(\xt{y},\xt{q}+t\xt{y}),\]
where 
\begin{equation}\label{caracexp}\Psi(\xt{q})=\int d\xt{p}\sigma(\xt{p})(e^{i\xt{p}\cdot\xt{q}}-1).\end{equation}

\end{thm}

In Theorem \ref{regularity}, $H^k(\mathbb{R}^{2d})$ stands for the $k$th Sobolev sapce on $\mathbb{R}^{2d}$. The proof of Theorem \ref{regularity} is given in Section \ref{proofthreg}. This result is an important consequence of the slowly decaying correlation assumption in time \eqref{SDCA} and cannot be observed in the case of rapidly decaying correlations in time \eqref{SDCAcor}. Theorem \ref{thasymptotic} and Theorem \ref{regularity}, leads us to the following corollary.

\begin{cor}\label{thasymptotic2}
The family $(W_\e)_{\e\in(0,1)}$, solution of the transport equation \eqref{wignereq} with $W_0\in L^2(\mathbb{R}^{2d})$, converges in probability on 
$\mathcal{C}([0,+\infty), (\mathcal{B}_{r} ,d_{\mathcal{B}_{r}}))$ as $\e \to 0$ to a limit denoted by $W$, which is the unique classical solution uniformly bounded in $L^2(\mathbb{R}^{2d})$ of the radiative transfer equation
\begin{equation}\label{radtranseq2}
\partial_t W+\xt{k}\cdot \nabla_{\xt{x}} W=\mathcal{L}W,
\end{equation}
with $W(0,\xt{x},\xt{k})=W_0(\xt{x},\xt{k})$, and where $\mathcal{L}$ is defined by \eqref{formgene}.

\end{cor}

Thanks to the long-range decorrelation assumption in time \eqref{SDCAreg}, the limit obtained in Theorem \ref{thasymptotic} in not only a weak solution of the radiative transfer equation \eqref{radtranseq} but also a classical solution. In Section \ref{probrepsec}, we give an explanation of this regularizing effect using a probabilistic representation of the solution of \eqref{radtranseq} in term of a Lévy process with jump measure $\sigma(\xt{p})d\xt{p}$.

\section{Probabilistic Representation}\label{probrepsec}

In this section we discuss the probabilistic representation of the solution of the radiative transfer equation \eqref{radtranseq} in term of Lévy process. A Lévy process is a stochastic process with independent and stationary increments. We refer to \cite{applebaum, sato} for the basic properties of the Lévy processes. A particular property of these processes is that they are entirely characterized by their characteristic exponent $\Psi$ defined by $\mathbb{E}[e^{i\xt{q}\cdot L_t}]=e^{-t\Psi(\xt{q})}$. For instance, there exists a Lévy process associated to the generator \eqref{formgene} and for which its characteristic exponent is given by $\Psi(\xt{q})=\int d\xt{p}\sigma(\xt{p})(e^{i\xt{p}\cdot\xt{q}}-1)$.  

\begin{prop}\label{probrep}
Let $W$ be the unique weak solution of the radiative transport equation \eqref{radtranseq} with initial datum $W_0\in L^2(\mathbb{R}^{2d})$. Then, there exists a Lévy process $(L_t)_{t\geq 0}$ with caracteristic exponent 
\eqref{caracexp} such that $L_0=0$, and
\[W(t, \xt{x},\xt{k})=\E\Big[W_0\Big(\xt{x}-t\xt{k}-\int_0^t L_sds,\xt{k}+L_t\Big)\Big].\]
According to Theorem \ref{regularity}, $W$ is the unique classical solution  uniformly bounded in $L^2(\mathbb{R}^{2d})$ of \eqref{radtranseq2}.
\end{prop} 

The proof of Proposition \ref{probrep} is given in Section \ref{proofproprep}. According to \eqref{caracexp}, the Lévy process $(L_t)_t$ is a pure jump process with jump measure $\sigma(\xt{p})d\xt{p}$. Therefore, because of the long-range correlation property \eqref{SDCAreg} the jump process $(L_t)_t$ has infinitely many small jumps. This last property of the symmetric process $(L_t)_t$ is the key to show the regularizing effect of the radiative transfer equation \eqref{radtranseq}. In fact, according to \cite[Proposition 1.1]{picard}  $L_t$ has a smooth bounded density, which permits to obtain smoothness to the variable $\xt{k}$. Moreover, the transport term in \eqref{radtranseq} permits to transfer the smoothness of the variable $\xt{k}$ to the variable $\xt{x}$. The probabilistic representation of the solution of \eqref{radtranseq} presented in  Proposition \ref{probrep} will be usefull in the proof of Theorem \ref{fractional}.

\section{Fractional Radiative Transfer Equation}\label{fractionalsec}

The goal of this section is to give an approximation of the radiative transfer equation \eqref{radtranseq}. This approximation permits to get a simpler radiative transfer model and derive an explicit formula of the solution, from which one can easily extract the exponent of the decaying power law.

The nonlocal effect of $\mathcal{L}$, defined by \eqref{formgene}, is characterized by the support of the transfer coefficient $\sigma(\xt{p})$ defined by \eqref{sigmadef}. The extension of this support is tantamount to the extension of the nonlocal effects of the transfer operator $\mathcal{L}$. In this section, we are interested in what does the radiative transfer equation \eqref{radtranseq} look like if the support of the transfer coefficient increase? 

Let us recall that the transfer coefficient is given by
\[\sigma(\xt{p})=\frac{\widehat{R}_0(\xt{p})}{\mathfrak{g}(\xt{p})}=\widehat{R}(0,\xt{p})=\int \E [V(t+s,\xt{x}+\xt{y})V(s,\xt{y})]e^{-i\xt{p}\cdot \xt{x}}d\xt{x}dt,\]
which is the power spectrum of the two-point correlation of the random potential $V$ at frequency $0$. In order to study the impact of the extension of the support of $\sigma(\xt{p})$, let us consider the power spectrum given by
\begin{equation}\label{scaling}\sigma^\eta(\xt{p})=\frac{1}{\eta^{1-\theta}}\int \E \Big[V\Big(\frac{t+s}{\eta},\frac{\xt{x}+\xt{y}}{\eta}\Big)V\Big(\frac{s}{\eta},\frac{\xt{y}}{\eta}\Big)\Big]e^{-i\xt{p}\cdot \xt{x}}d\xt{x}dt=\eta^{d+\theta}\sigma(\eta\xt{p}).\end{equation}
Here, the parameter $\eta$ represents the extension of the support of $\sigma^\eta(\xt{p})$, and therefore the extension of the nonlocal effects of the corresponding transfer operator $\mathcal{L}^\eta$ as $\eta$ goes to $0$. Let us remark that this scaling corresponds to the long space and time diffusion approximation of the radiative transfer equation \eqref{radtranseq}.

In this paper, we have assumed \eqref{SDCAreg}, which means that that the transfer coefficient behaves locally at $0$ like the transfer coefficient of the fractional Laplacian $-(-\Delta)^{\theta/2}$. Thanks to the scaling \eqref{scaling} this local behavior becomes global as shown in Theorem \ref{fractional}. In fact, using \eqref{SDCAreg} we have
\begin{equation}\label{globalfrac}
\forall \xt{p}\in \mathbb{R}^d\setminus\{0\}, \quad \lim_{\eta\to 0} \sigma^\eta(\xt{p})=\frac{\sigma}{\lvert \xt{p}\rvert^{d+\theta}}.
\end{equation}
Moreover, let us assuming that $\widehat{R}_0$ is bounded on each compact of $\mathbb{R}^d\setminus\{0\}$, we also have
\begin{equation}\label{boundedfrac}
0\leq \sigma^\eta(\xt{p})\leq \frac{C}{\lvert\xt{p} \rvert^{d+\theta}}\quad \forall \xt{p}\in \mathbb{R}^{d}\setminus \{0\},
\end{equation}
with $C>0$.
\begin{thm}\label{fractional}
Let us assume that  $\widehat{R}_0$ is bounded on each compact of $\mathbb{R}^d\setminus\{0\}$. Then, we have
\[\forall W_0\in L^2(\mathbb{R}^{2d})\text{ and }\forall t\geq 0, \quad \lim_{\eta\to 0} W^\eta(t,\xt{x},\xt{k})=W^\infty(t,\xt{x},\xt{k})\]
pointwise and weakly in $L^2(\mathbb{R}^{2d})$. Here, $W^\eta$ is the unique classical solution uniformly bounded in $L^2(\mathbb{R}^{2d})$ of the radiative transfer equation \eqref{radtranseq}, and $W^\infty$ is the unique weak solution uniformly bounded in $L^2(\mathbb{R}^{2d})$ of the following radiative transfer equation
\begin{equation}\label{asympradeq}\partial_t W^\infty+\xt{k}\cdot\nabla_{\xt{x}}W^\infty=-\sigma(\theta)(-\Delta_\xt{k})^{\theta/2}W^{\infty}, \quad W^\infty(0,\xt{x},\xt{k})=W_0(\xt{x},\xt{k}).\end{equation}
Here, $(-\Delta_\xt{k})^{\theta/2}$ is the fractional Laplacian with Hurst index $\theta\in(0,1)$, and
\[\sigma(\theta)=\frac{\sigma \theta\Gamma(1-\theta)}{(2\pi^d)}\int_{\mathbb{S}^{d-1}}dS(\xt{u})\lvert\xt{e}_1\cdot \xt{u} \rvert^\theta\]
with $\xt{e}_1\in \mathbb{S}^{d-1}$ and $\Gamma(z)=\int_0^{+\infty} t^{1-z}e^{-t}dt$.
Moreover, $W^\infty$ is given by the following formula
\begin{equation}\label{formula}W^\infty(t,\xt{x},\xt{k})=\frac{1}{(2\pi)^d}\int d\xt{y}d\xt{q} e^{i(\xt{x}\cdot\xt{y}+\xt{k}\cdot\xt{q})}e^{-\sigma(\theta)\int_0^t du \lvert \xt{q}+u\xt{y} \rvert^\theta}\widehat{W}_0(\xt{y},\xt{q}+t\xt{y}).\end{equation}
\end{thm}

The proof of Theorem \ref{fractional} is given in Section \ref{proofthfrac} and uses the probabilistic representation obtained in Proposition \ref{probrep}. 

Equation \eqref{asympradeq} is an approximation of \eqref{radtranseq} when we extend the nonlocal effect, and it permits to get an explicit formula \eqref{formula} of its solutions. This formula permits to exhibit a damping term, which means that the phase space energy distribution decays with respect to time. Moreover this damping term obeying to a power law with exponent $\theta\in(0,1)$. In wave propagation such a kind of result result has been already obtained for the wave equation in one dimensional propagation medium \cite{garnier}, and physical models involving fractional Laplacian have been proposed to predict such a power law \cite{chen}. 

Regarding the regularity of $W^\infty$ we can prove the following proposition in the same way as Theorem \ref{regularity}. This result states that the unique solution of \eqref{asympradeq} $W^\infty$ is actually a smooth function which is the unique classical solution of \eqref{asympradeq}.

\begin{prop}
Let $W_0\in L^2 (\mathbb{R}^{2d})$ and $W^\infty$ be the unique solution of the radiative transfer equation \eqref{asympradeq}. Then, $\forall t_0>0$, we have 
\[
W^\infty\in \mathcal{C}^0\Big( (0,+\infty), \bigcap_{k\geq 0} H^k (\mathbb{R}^{2d}) \Big)\cap L^\infty \Big([t_0,+\infty),\bigcap _{k\geq 0} H^k (\mathbb{R}^{2d})\Big),
\]
so that we also have 
\[W^\infty\in \mathcal{C}^0\big( (0,+\infty),\mathcal{C}^\infty(\mathbb{R}^{2d})\big)\cap L^\infty \big( [t_0,+\infty),\mathcal{C}^\infty(\mathbb{R}^{2d})\big).\] 
Then, $W^\infty$ is the unique classical solution uniformly bounded in $L^2(\mathbb{R}^{2d})$ of the radiative transfer equation
\[\partial_t W^\infty+\xt{k}\cdot\nabla_{\xt{x}}W^\infty=-\sigma(\theta)(-\Delta)^{\theta/2}W^{\infty}, \quad W^\infty(0,\xt{x},\xt{k})=W_0(\xt{x},\xt{k}).\]
\end{prop}

\section{Proof of Theorem \ref{thasymptotic}}\label{proofthlimit}

The proof of Theorem \ref{thasymptotic} is based on the perturbed-test-function approach. Using the notion of a pseudogenerator, we prove tightness and characterize all subsequence limits.

Using a particular tightness criteria, we prove the tightness of the family $(W_\e)_{\e \in (0,1)}$ on the polish space $\mathcal{C}([0,+\infty),(\mathcal{B}_{r},d_{\mathcal{B}_{r}}))$. In the next section, we characterize all subsequence limits as weak solutions of a well-posed radiative transfer equation. 

We have the following version of the Arzelà-Ascoli theorem \cite{billingsley,kallianpur} for processes with values in a complete separable metric space. 
\begin{thm}
A set $B\subset \mathcal{C}([0,+\infty),(\mathcal{B}_{r},d_{\mathcal{B}_{r}} ))$ has a compact closure if and only if 
\[\forall T>0, \quad \lim_{\eta \to 0}\sup_{g \in A} m_T (g, \eta)=0,  \]
with 
\[m_T (g,\eta)=\sup_{\substack{(s,t)\in [0,T]^2\\ \lvert t-s \rvert \leq\eta}} d_{\mathcal{B}_{r}}( g(s) ,g(t)). \] 
\end{thm}
From this result, we obtain the classical tightness criterion. 
\begin{thm}
A family of probability measure $\big(\mathbb{P}^\e\big)_{\e\in(0,1)}$ on $\mathcal{C}([0,+\infty),(\mathcal{B}_{r},d_{\mathcal{B}_{r}} ))$ is tight if and only if
\[\forall T>0, \eta'>0 \quad \lim_{\eta \to 0}\sup_{\e \in (0,1) }\mathbb{P}^\e \big(g\,; \,\,m_T (g, \eta)>\eta'\big)=0.\]
\end{thm}
From the definition of the metric $d_{\mathcal{B}_{r}}$, the tightness criterion becomes the following.
\begin{thm}\label{crit}
A family of processes $(X^\e)_{\e\in(0,1)}$ is tight on $\mathcal{C}([0,+\infty),(\mathcal{B}_{r},d_{\mathcal{B}_{r}} ))$ if and only if the process $\big(\big<X^\e,\lambda\big>_{L^2(\mathbb{R}^d\times\mathbb{R}^d)}\big)_{\e\in(0,1)}$ is tight on $\mathcal{C}([0,+\infty),\mathbb{R})$ $\forall \lambda \in L^2(\mathbb{R}^{2d})$.
\end{thm}
This last theorem looks like the tightness criterion of Mitoma and Fouque \cite{mitoma,fouque}.

For any $\lambda \in L^2(\mathbb{R}^{2d})$, we set $W_{\e,\lambda}(t)=\big<W_{\e}(t),\lambda\big>_{L^2(\mathbb{R}^{2d})}$. According to Theorem \ref{crit}, the family $(W_{\e})_\e$ is tight on $\mathcal{C}([0,+\infty), (\mathcal{B}_{r},d_{\mathcal{B}_{r}} ))$ if and only if the family $(W_{\e,\lambda})_\e$ is tight on $\mathcal{C}([0,+\infty),\mathbb{R} )$,  $\forall \lambda \in L^2(\mathbb{R}^{2d})$. Furthermore, $(W_\e)_{\e}$ is a family of continuous processes. Then, according to \cite[Theorem 13.4]{billingsley}, it is sufficient to prove that, $\forall\lambda \in L^2(\mathbb{R}^{2d})$, $(W_{\e,\lambda})_{\e}$ is tight on $\mathcal{D}([0,+\infty),\mathbb{R} )$, which is the set of cad-lag functions with values in $\mathbb{R}$. Finally, using that the process $W_\e$ is a process with values in $\mathcal{B}_r$ and that the set of all smooth functions with compact support in $\mathbb{R}^{2d}$ is dense in $L^2(\mathbb{R}^{2d})$, it is sufficient to show that $(W_{\e,\lambda})_{\e}$ is tight on $\mathcal{D}([0,+\infty),\mathbb{R} )$ $\forall\lambda \in \mathcal{C}^{\infty}_0(\mathbb{R}^{2d})$.

\subsection{Pseudogenerator}\label{pseudogene}

We recall the techniques developed by Kurtz and Kushner \cite{kushner}. Let $\mathcal{M}^\e $ be the set of all $\mathcal{F}^\e$-measurable functions $f(t)$ for which $\sup_{t\leq T} \mathbb{E}\left[\lvert f(t) \rvert \right] <+\infty $ and where $T>0$ is fixed. Here, $\mathcal{F}^\e _t =\mathcal{F}_{t/\e}$ and $(\mathcal{F}_{t})$ is defined by \eqref{filtration}. The $p-\lim$ and the pseudogenerator are defined as follows. Let $f$ and $f^\delta$ in $\mathcal{M}^\e $ $\forall\delta>0$. We say that $f=p-\lim_\delta f^\delta$ if
\[ \sup_{t, \delta }\mathbb{E}[\lvert f^\delta(t)\rvert]<+\infty\quad \text{and}\quad \lim_{\delta\rightarrow 0}\mathbb{E}[\lvert f^\delta (t) -f(t)\rvert]=0 \quad \forall t.\]
The domain of $\mathcal{A}^\e$ is denoted by $\mathcal{D}\left(\mathcal{A}^\e\right)$. We say that $f\in \mathcal{D}\left(\mathcal{A}^\e\right)$ and $\mathcal{A}^\e f=g$ if $f$ and $g$ are in $\mathcal{D}\left(\mathcal{A}^\e\right)$ and 
\begin{equation*}
p-\lim_{\delta \to 0} \left[ \frac{\mathbb{E}^\e _t [f(t+\delta)]-f(t)}{\delta}-g(t) \right]=0,
\end{equation*}
where $\mathbb{E}^\e _t$ is the conditional expectation given $\mathcal{F}^\e _t$. A useful result about pseudogenerator $\mathcal{A}^\e$ is given by the following theorem.
\begin{thm}\label{martingale}
Let $f\in \mathcal{D}\left(\mathcal{A}^\e\right)$. Then
\begin{equation*}
M_f ^\e (t)=f(t)-f(0)-\int _0 ^t  \mathcal{A}^\e f(u)du
\end{equation*}
is an $\left( \mathcal{F}^\e _t \right)$-martingale.
\end{thm}

\subsection{Tightness}\label{tightnesssec}

\begin{prop}\label{tightness}
$\forall \lambda \in\mathcal{C}^{\infty}_0(\mathbb{R}^{2d})$, the family $(W_{\e,\lambda})_{\e\in(0,1)}$ is tight on $\mathcal{D}\left([0,+\infty),\mathbb{R} \right)$.
\end{prop}
\begin{proof}
According to \cite[Theorem 4]{kushner} and because $W_\e$ is a process with values in $\mathcal{B}_r$, we need to show Lemma \ref{bound2} and Lemma \ref{A1} below. Throughout the proof Propostion \ref{tightness}, let $\lambda \in \mathcal{C}^{\infty}_0(\mathbb{R}^{2d})$, $f$ be a bounded smooth function, and $f_0 ^\e (t)=f(W_{\e,\lambda} (t))$. 
\begin{lem} \label{bound1}
$\forall T>0$,
\[\E\big[\sup_{t\in[0,T]}\|\mathcal{L}_\e\lambda(t)\|^2_{L^2(\mathbb{R}^{2d})}\big]<+\infty. \]
\end{lem}
\begin{proof}[of Lemma \ref{bound1}]
First,
\[ \|\mathcal{L}_\e\lambda(t)\|^2_{L^2(\mathbb{R}^{2d})}\leq \frac{1}{(2\pi)^d}\int_{-1/2}^{1/2}du \int d\xt{x}d\xt{k} \Big\lvert \int \widehat{V}\Big(\frac{t}{\e^{1+\ga}},d\xt{p}\Big)e^{i\xt{p}\cdot\xt{x}/\e}\xt{p}\cdot \nabla_{\xt{k}} \lambda(\xt{x},\xt{k}+u\xt{p})  \Big\rvert^2. \]
Let us fixe $\xt{x}$, $\xt{k}$, and $u$. Let
\begin{equation}\label{philambda}
\phi_{1,\lambda,\xt{x},\xt{k},u}(\xt{p})=e^{i\xt{p}\cdot\xt{x}/\e}\xt{p}\cdot\nabla_{\xt{k}}\lambda(\xt{x},\xt{k}+u\xt{p}).\end{equation}
It is clear that $\phi_{1,\lambda,\xt{x},\xt{k},u}\in\mathcal{H}$. Consequently, $V_1=\big<\widehat{V},\phi_{\lambda,\xt{x},\xt{k},u}\big>_{\mathcal{H}',\mathcal{H}}$ is centered Gaussian process with a pseudo-metric $m_1$ on $[0,T]$ given by 
\[m_1(t,s)=\mathbb{E}\left[\Big(V_1\Big(\frac{t}{\e^{1+\ga}}\Big) -V_1\Big(\frac{s}{\e^{1+\ga}}\Big)\Big)^2\right]^{1/2}.\]
Then,
\[m^2_1(t,s)\leq 2 \sup_{\xt{x},\xt{k}}\lvert \nabla_{\xt{k}}\lambda (\xt{x},\xt{k})\rvert^2  \left(\int d\xt{p}\widehat{R}_0(\xt{p}) \mathfrak{g}(\xt{p})\lvert \xt{p}\rvert^2\right)\frac{\lvert t-s\rvert}{\e^{1+\ga}} \quad \forall(t,s)\in[0,T]^2,\]
and
\[diam^2_{m_1}([0,T]) \leq 2\int \widehat{R}_0(\xt{p})\lvert \xt{p}\rvert^2 \lvert \nabla_{\xt{k}}\lambda(\xt{x}, \xt{k}+u\xt{p})\rvert^2. \]
Here, $diam_{m_1}([0,T])$ stands for the diameter of $[0,T]$ under the pseudo-metric $m_1$. According to \cite[Theorem 2.1.3]{adlertaylor}, we have
\[\begin{split}
\mathbb{E}\left[\sup_{t\in[0,T]}\Big \lvert V_1\Big(\frac{t}{\e^{1+\ga}}\Big)\Big\rvert ^2  \right]&\leq K \left(\int_0^{diam_{m_1}( [0,T])/2}H^{1/2}(r)dr\right)^2\\
&\leq C_1\left(\int_0^{\theta_1(\xt{x},\xt{k},u)} \sqrt{\ln\left( C_2\frac{T}{r^2\e^{1+\ga}}  \right)}dr\right)^2,
\end{split} \]
where $H(r)=\ln(N(r))$, and $N(r)$ denotes the smallest number of balls, for the pseudo-metric $m_1$, with radius $r$ to cover $[0,T]$. Moreover, $\theta_1$ is given by
\[\theta_1^2(\xt{x},\xt{k},u)=2\int \widehat{R}_0(\xt{p})\lvert \xt{p}\rvert^2 \lvert \nabla_{\xt{k}}\lambda(\xt{x}, \xt{k}+u\xt{p})\rvert^2.\] 
Consequently, 
\[\int d\xt{x}d\xt{k} \left(\int_0^{\theta_1(\xt{x},\xt{k},u)} \sqrt{\ln\left( C_2\frac{T}{r^2\e^{1+\ga}}  \right)}dr\right)^2 <+\infty,\]
since $\widehat{R}_0$ and $\lambda$ have a support included in a compact set, and that concludes the proof of this lemma.$\square$
\end{proof}
Thanks to Lemma \ref{bound1}, we can use the notion of pseudogenerator introduced in Section \ref{pseudogene}. Therefore, we have
\[
\mathcal{A}^\e f_0 ^\e (t)= f'(W_{\e,\lambda} (t))\left[W_{\e,\lambda_1}(t)+\frac{1}{\e^{\frac{1+\ga}{2}}}\big<W_\e(t),\mathcal{L}_\e(t)\lambda\big>_{L^2(\mathbb{R}^{2d})}\right] 
\]
where $\mathcal{L}_\e$ is defined by \eqref{lepsilon}, and 
\begin{equation}\label{lambda1}
\lambda_1(\xt{x},\xt{k})=\xt{k}\cdot\nabla_{\xt{x}}\lambda(\xt{x},\xt{k}).
\end{equation}
Let
\[
\begin{split}
f^\e _1 (t)=&\frac{1}{\e^{\frac{1+\ga}{2}}} f'(W_{\e,\lambda} (t))\int d\xt{x}d\xt{k}W_\e(t,\xt{x},\xt{k})\\
&\hspace{0.5cm}\times\int_{t}^{+\infty} \mathbb{E}^\e _t\Big[\int\frac{\widehat{V}\big(\frac{u}{\e^{1+\ga}},d\xt{p}\big)}{(2\pi)^d i}e^{i(u-t)\xt{p}\cdot \xt{k}/\e}e^{i\xt{p}\cdot\xt{x}/\e}\times\Big(\lambda\big(\xt{x},\xt{k}-\frac{\xt{p}}{2}\big)-\lambda\big(\xt{x},\xt{k}+\frac{\xt{p}}{2}\big)\Big)\Big] du. 
\end{split}\]
\begin{lem}\label{bound2}  $\forall T>0$, and $\eta>0$
\[\lim_{\e} \Pro\Big(\sup_{0\leq t \leq T} \lvert f^\e _1 (t)\rvert>\eta\Big) =0,\quad \text{and} \quad\sup_{t\geq 0}\mathbb{E}\left[\lvert f^\e _1 (t) \rvert \right]=\mathcal{O}(\e^{(1-\ga)/2}).\]
\end{lem}
\begin{proof}[of Lemma \ref{bound2}]
Using \eqref{markovesp}, we have
\[f^{\e}_1(t)= \e^{\frac{1+\ga}{2}} f'(W_{\e,\lambda} (t))\big<W_\e(t),\mathcal{L}_{1,\e}\lambda(t)\big>_{L^2(\mathbb{R}^{2d})},\]
with
\[
\mathcal{L}_{1,\e}\lambda(t,\xt{x},\xt{k})=\frac{1}{(2\pi)^d i}\int\frac{\widehat{V}\big(\frac{t}{\e^{1+\ga}},d\xt{p}\big)}{\mathfrak{g}(\xt{p})-i\e^{\ga}\xt{k}\cdot\xt{p}}e^{i\xt{p}\cdot\xt{x}/\e}\left(\lambda\big(\xt{x},\xt{k}-\frac{\xt{p}}{2}\big)-\lambda\big(\xt{x},\xt{k}+\frac{\xt{p}}{2}\big)\right).
\]

\begin{lem}\label{bound3}
\[\E\left[\sup_{t\in[0,T]}\|\mathcal{L}_{1,\e}\lambda(t)\|^2_{L^2(\mathbb{R}^{2d})} \right]<+\infty\quad\text{and}\quad \lim_\e \e^{\frac{1+\ga}{2}} \E\left[\sup_{t\in[0,T]}\|\mathcal{L}_{1,\e}\lambda(t)\|^2_{L^2(\mathbb{R}^{2d})} \right]=0.\]
\end{lem}
\begin{proof}[of Lemma \ref{bound3}]
Here, we use the same argument as in Lemma \ref{bound1}. First,
\[ \|\mathcal{L}_{1,\e}\lambda(t)\|^2_{L^2(\mathbb{R}^{2d})}\leq \frac{1}{(2\pi)^d}\int_{-1/2}^{1/2}du \int d\xt{x}d\xt{k} \Big\lvert \int \frac{\widehat{V}\big(\frac{t}{\e^{1+\ga}},d\xt{p}\big)}{\mathfrak{g}(\xt{p})-i\e^{\ga}\xt{k}\cdot\xt{p}}e^{i\xt{p}\cdot\xt{x}/\e}\xt{p}\cdot \nabla_{\xt{k}} \lambda(\xt{x},\xt{k}+u\xt{p})  \Big\rvert^2. \]
Let us fixe $\xt{x}$, $\xt{k}$, and $u$. Let
\begin{equation}\label{philambda1}\phi_{2,\lambda,\xt{x},\xt{k},u}(\xt{p})=\frac{e^{i\xt{p}\cdot\xt{x}/\e}}{\mathfrak{g}(\xt{p})-i\e^{\ga}\xt{k}\cdot\xt{p}}\xt{p}\cdot\nabla_{\xt{k}}\lambda(\xt{x},\xt{k}+u\xt{p}).\end{equation}
According to \eqref{longrange} $\phi_{2,\xt{x},\xt{k},u}\in\mathcal{H}$. Consequently, $V_2=\big<\widehat{V},\phi_{2,\lambda,\xt{x},\xt{k},u}\big>_{\mathcal{H}',\mathcal{H}}$ is centered Gaussian process with a pseudo-metric $m_2$ on $[0,T]$ given by 
\[m_2(t,s)=\mathbb{E}\left[\Big(V_2\Big(\frac{t}{\e^{1+\ga}}\Big) -V_2\Big(\frac{s}{\e^{1+\ga}}\Big)\Big)^2\right]^{1/2}.\]
Then,
\[m^2_2(t,s)\leq 2 \sup_{\xt{x},\xt{k}}\lvert \nabla_{\xt{k}}\lambda (\xt{x},\xt{k})\rvert^2  \left(\int d\xt{p}\widehat{R}_0(\xt{p}) \frac{\mathfrak{g}(\xt{p})\lvert \xt{p}\rvert^2}{\mathfrak{g}^2(\xt{p})}\right)\frac{\lvert t-s\rvert}{\e^{1+\ga}} \quad \forall(t,s)\in[0,T]^2,\]
and
\[diam^2_{m_2}([0,T]) \leq 2\int \widehat{R}_0(\xt{p})\frac{\lvert \xt{p}\rvert^2}{\mathfrak{g}^2(\xt{p})} \lvert \nabla_{\xt{k}}\lambda(\xt{x}, \xt{k}+u\xt{p})\rvert^2. \]
Here, $diam_{m_2}([0,T])$ stands for the diameter of $[0,T]$ under the pseudo-metric $m_2$. According to \cite[Theorem 2.1.3]{adlertaylor}, we have
\[\mathbb{E}\left[\sup_{t\in[0,T]}\Big\lvert V_2\Big(\frac{t}{\e^{1+\ga}}\Big) \Big\rvert^2 \right]\leq C_1\left(\int_0^{\theta_2(\xt{x},\xt{k},u)} \sqrt{\ln\left( C_2\frac{T}{r^2\e^{1+\ga}}  \right)}dr\right)^2,\]
where
\[\theta_2^2(\xt{x},\xt{k},u)=2\int \widehat{R}_0(\xt{p})\frac{\lvert \xt{p}\rvert^2}{\mathfrak{g}^2(\xt{p})} \lvert \nabla_{\xt{k}}\lambda(\xt{x}, \xt{k}+u\xt{p})\rvert^2.\] 
Consequently, 
\[\int d\xt{x}d\xt{p} \left(\int_0^{\theta_2(\xt{x},\xt{k},u)} \sqrt{\ln\left( C_2\frac{T}{r^2\e^{1+\ga}}  \right)}dr\right)^2<+\infty,\]
since $\widehat{R}_0$ and $\lambda$ have a support included in a compact set, and that concludes the proof of Lemma \ref{bound3}.$\square$
\end{proof}
Then, the proof of Lemma \ref{bound2} is a direct consequence of Lemma \ref{bound3}.$\square$
\end{proof}

The following lemma is a consequence of Lemma \ref{bound3}, Lemma \ref{bound4}, and Lemma \ref{bound5}.
\begin{lem}\label{A1}
$\forall T>0$, $\left\{\mathcal{A}^\e \left(f^\e _0 +f^\e _1\right)(t), \e \in(0,1), 0\leq t\leq T\right\}$ is uniformly integrable.
\end{lem}
\begin{proof}[of Lemma \ref{A1}]
First, let us show that we can compute the pseudogenerator at $f^\e _0 +f^\e _1$.

\begin{lem}\label{bound4}
\[\sup_{\e\in(0,1)}\E\left[\|\mathcal{L}_{\e}\big( \mathcal{L}_{1,\e}\lambda(t)\big)(t)\|^2_{L^2(\mathbb{R}^{2d})} \right]<+\infty.\]
\end{lem}
\begin{proof}[of Lemma \ref{bound4}]
We have
\[\begin{split}
\mathcal{L}_{\e}\big( \mathcal{L}_{1,\e}\lambda(t)\big)&(t,\xt{x},\xt{k})=\\
&\iint\widehat{V}\big(\frac{t}{\e^{1+\ga}},d\xt{p}_1\big)\widehat{V}\big(\frac{t}{\e^{1+\ga}},d\xt{p}_2\big)e^{i(\xt{p}_1+\xt{p}_2)\cdot \xt{x}/\e}\\
&\hspace{0.5cm}\times\Big(\frac{1}{\mathfrak{g}(\xt{p}_2)-i\e^{\ga}\big(\xt{k}-\frac{\xt{p}_1}{2}\big)\cdot\xt{p}_2}\big(\lambda\big(\xt{x},\xt{k}-\frac{\xt{p}_1}{2}-\frac{\xt{p}_2}{2}\big)-\lambda\big(\xt{x},\xt{k}-\frac{\xt{p}_1}{2}+\frac{\xt{p}_2}{2}\big) \big)\\
&\hspace{1cm}-\frac{1}{\mathfrak{g}(\xt{p}_2)-i\e^{\ga}\big(\xt{k}+\frac{\xt{p}_1}{2}\big)\cdot\xt{p}_2}\big(\lambda\big(\xt{x},\xt{k}+\frac{\xt{p}_1}{2}-\frac{\xt{p}_2}{2}\big)-\lambda\big(\xt{x},\xt{k}+\frac{\xt{p}_1}{2}+\frac{\xt{p}_2}{2}\big) \big)\Big).
\end{split}\]
Let us note that 
\begin{equation}\label{vargaussien}
\begin{split}
\E\big[\widehat{V}(t_1,d\xt{p}_1)\widehat{V}(t_2,d\xt{p}_2)&\widehat{V}^\ast(t_3,d\xt{p}_3)\widehat{V}^\ast(t_4,d\xt{p}_4)\big]=\\
&\hspace{0.4cm}(2\pi)^{2d}\tilde{R}(t_1-t_2,\xt{p}_1)\tilde{R}(t_3-t_4,\xt{p}_3)\delta(\xt{p}_1+\xt{p}_2)\delta(\xt{p}_3+\xt{p}_4)\\
&+(2\pi)^{2d}\tilde{R}(t_1-t_3,\xt{p}_1)\tilde{R}(t_2-t_4,\xt{p}_3)\delta(\xt{p}_1-\xt{p}_3)\delta(\xt{p}_2-\xt{p}_4)\\
&+(2\pi)^{2d}\tilde{R}(t_1-t_4,\xt{p}_1)\tilde{R}(t_2-t_3,\xt{p}_3)\delta(\xt{p}_1-\xt{p}_4)\delta(\xt{p}_2-\xt{p}_3).
\end{split}
\end{equation}
Then, using the smoothness of $\lambda$ and \eqref{longrange}, we obtain
\[\begin{split}
\E\big[\|\mathcal{L}_{\e}\big(& \mathcal{L}_{1,\e}\lambda(t)\big)(t)\|^2_{L^2(\mathbb{R}^{2d})} \big]\leq C \iint d\xt{p}_1 d\xt{p}_2\widehat{R}_0(\xt{p}_1)\widehat{R}_0(\xt{p}_2)\\
&\times\Big[\Big(\frac{\lvert \xt{p}_1\rvert^3}{\mathfrak{g}(\xt{p}_1)^2} \sup_{\xt{x},\xt{k}}\lvert \nabla_{\xt{k}}\lambda(\xt{x},\xt{k})\rvert+\frac{\lvert\xt{p}_1\rvert^2 }{\mathfrak{g}(\xt{p}_1)}\sup_{\xt{x},\xt{k}}\|D^2\lambda(\xt{x},\xt{k})\|  \Big) \\
&\hspace{0.5cm}\times\Big(\frac{\lvert \xt{p}_2\rvert^3}{\mathfrak{g}(\xt{p}_2)^2} \sup_{\xt{x},\xt{k}}\lvert \nabla_{\xt{k}}\lambda(\xt{x},\xt{k})\rvert+\frac{\lvert\xt{p}_2\rvert^2 }{\mathfrak{g}(\xt{p}_2)}\sup_{\xt{x},\xt{k}}\|D^2\lambda(\xt{x},\xt{k})\|  \Big) \\
&+\Big(\frac{\lvert \xt{p}_1\rvert\lvert \xt{p}_2\rvert^2}{\mathfrak{g}^2(\xt{p}_2)}\sup_{\xt{x},\xt{k}}\lvert \nabla_{\xt{x}}\lambda(\xt{x},\xt{k})\rvert+\frac{1}{\mathfrak{g}(\xt{p}_2)}(\lvert \xt{p}_1\rvert\lvert \xt{p}_2\rvert+\lvert \xt{p}_2\rvert^2)\sup_{\xt{x},\xt{k}}\|D^2\lambda(\xt{x},\xt{k})\|\Big)^2\\
&+\Big(\frac{\lvert \xt{p}_1\rvert\lvert \xt{p}_2\rvert^2}{\mathfrak{g}^2(\xt{p}_2)}\sup_{\xt{x},\xt{k}}\lvert \nabla_{\xt{x}}\lambda(\xt{x},\xt{k})\rvert+\frac{1}{\mathfrak{g}(\xt{p}_2)}(\lvert \xt{p}_1\rvert\lvert \xt{p}_2\rvert+\lvert \xt{p}_2\rvert^2)\sup_{\xt{x},\xt{k}}\|D^2\lambda(\xt{x},\xt{k})\|\Big)\\
&\hspace{0.5cm}\times\Big(\frac{\lvert \xt{p}_2\rvert\lvert \xt{p}_1\rvert^2}{\mathfrak{g}^2(\xt{p}_1)}\sup_{\xt{x},\xt{k}}\lvert \nabla_{\xt{x}}\lambda(\xt{x},\xt{k})\rvert+\frac{1}{\mathfrak{g}(\xt{p}_1)}(\lvert \xt{p}_2\rvert\lvert \xt{p}_1\rvert+\lvert \xt{p}_1\rvert^2)\sup_{\xt{x},\xt{k}}\|D^2\lambda(\xt{x},\xt{k})\|\Big)\Big]<+\infty,
\end{split}\]
since $\lambda$ and $\widehat{R}_0$ have a support included in a compact set. That concludes the proof of Lemma \ref{bound4}.
$\square$.
\end{proof}

\begin{lem}\label{bound5}
\[\sup_{\e\in(0,1)}\E\big[\|\mathcal{L}_{\e}\lambda(t)\|^2_{L^2(\mathbb{R}^{2d})} \times\|\mathcal{L}_{1,\e}\lambda(t)\|^2_{L^2(\mathbb{R}^{2d})}\big]<+\infty.\]
\end{lem}
\begin{proof}[of Lemma \ref{bound5}]
This result follows from the temporal stationarity of the process $(\widehat{V}(t,\cdot))_{t}$. We have
\[\begin{split}
\E\big[\|\mathcal{L}_{\e}\lambda(t)&\|^2_{L^2(\mathbb{R}^d\times\mathbb{R}^d)}\times\|\mathcal{L}_{1,\e}\lambda(t)\|^2_{L^2(\mathbb{R}^d\times\mathbb{R}^d)}\big]\\
&\leq \iint d\xt{x}_1 d\xt{k}_1d\xt{x}_2d\xt{k}_2 \int_{-1/2}^{1/2}du_1\int_{-1/2}^{1/2}du_2\\
&\hspace{0.5cm}\E\big[\lvert \big<\widehat{V}(t/\e^{1+\ga}),\phi_{1,\lambda,\xt{x}_1,\xt{k}_1,u_1}\big>_{\mathcal{H}',\mathcal{H}}\rvert^2 \lvert\big<\widehat{V}(t/\e^{1+\ga}),\phi_{2,\lambda,\xt{x}_2,\xt{k}_2,u_2}\big>_{\mathcal{H}',\mathcal{H}} \rvert^2 \big]\\
& \leq \iint d\xt{x}_1 d\xt{k}_1d\xt{x}_2d\xt{k}_2 \int_{-1/2}^{1/2}du_1\int_{-1/2}^{1/2}du_2\\
&\hspace{0.5cm}\Big(\E\big[\lvert \big<\widehat{V}(t/\e^{1+\ga}),\phi_{1,\lambda,\xt{x}_1,\xt{k}_1,u_1}\big>_{\mathcal{H}',\mathcal{H}}\rvert^4\big]\Big)^{1/2}\Big(\E\big[ \lvert\big<\widehat{V}(t/\e^{1+\ga}),\phi_{2,\lambda,\xt{x}_2,\xt{k}_2,u_2}\big>_{\mathcal{H}',\mathcal{H}} \rvert^4 \big]\Big)^{1/2},
\end{split}\]
where $\phi_{1,\lambda,\xt{x},\xt{k},u}$ and $\phi_{2,\lambda,\xt{x},\xt{k},u}$ are defined respectively by \eqref{philambda} and \eqref{philambda1}. Moreover,
\[\begin{split}\E\big[\lvert \big<\widehat{V}(t/\e^{1+\ga}),\phi_{1,\lambda,\xt{x}_1,\xt{k}_1,u_1}\big>_{\mathcal{H}',\mathcal{H}}\rvert^4\big]&=3\E\big[\lvert \big<\widehat{V}(0),\phi_{1,\lambda,\xt{x},\xt{k},u}\big>_{\mathcal{H}',\mathcal{H}}\rvert^2\big]^2\\
&\leq 3\Big[\int\widehat{R}_0(\xt{p})\lvert \xt{p}\rvert^2 \lvert \nabla_{\xt{k}}\lambda(\xt{x},\xt{k}+u\xt{p})\rvert^2\Big]^2<+\infty,\end{split}\]
and thanks to $\eqref{longrange}$
\[\begin{split}
\E\big[\lvert \big<\widehat{V}(t/\e^{1+\ga}),\phi_{2,\lambda,\xt{x}_1,\xt{k}_1,u_1}\big>_{\mathcal{H}',\mathcal{H}}\rvert^4\big]&=3\E\big[\lvert \big<\widehat{V}(0),\phi_{2,\lambda,\xt{x},\xt{k},u}\big>_{\mathcal{H}',\mathcal{H}}\rvert^2\big]^2\\
&\leq 3\Big[\int\widehat{R}_0(\xt{p})\frac{\lvert \xt{p}\rvert^2}{\mathfrak{g}^2(\xt{p})} \lvert \nabla_{\xt{k}}\lambda(\xt{x},\xt{k}+u\xt{p})\rvert^2\Big]^2<+\infty.\end{split}\]
That conclude the proof of Lemma \ref{bound5}, since $\lambda$ and $\widehat{R}_0$ have a support included in a compact set.$\square$
\end{proof}

Consequently, thanks to Lemma \ref{bound3}, Lemma \ref{bound4}, and Lemma \ref{bound5}, we have
\[\begin{split}
\mathcal{A}^\e(f^\e _0 +f^\e _1)(t)=&f'(W_{\e,\lambda}(t))\left[W_{\e,\lambda_1}(t)+\big<W_\e(t),\mathcal{L}_{\e}\big(\mathcal{L}_{1,\e}\lambda(t)\big)(t)\big>_{L^2(\mathbb{R}^{2d})}\right]\\
&+\partial^2_v f(W_{\e,\lambda}(t))\big<W_\e(t),\mathcal{L}_{1,\e}\lambda(t)\big>_{L^2(\mathbb{R}^{2d})}\big<W_\e(t),\mathcal{L}_{\e}\lambda(t)\big>_{L^2(\mathbb{R}^{2d})}\\
&+\mathcal{O}(\e^{(1+\ga)/2}),
\end{split}\] 
and $\sup_{\e,t}\E[\lvert \mathcal{A}(f^\e_0+f^\e_1)(t)\rvert^2]<+\infty$. That conclude the proof of Lemma \ref{A1} and then the proof of Proposition \ref{tightness}. $\square$ $\blacksquare$
\end{proof}  

\end{proof}

\subsection{Identification of all subsequence limits}\label{identificationsec}

To identify all the subsequence limits of the process $(W_\e)_\e$, we show that all such limit processes are a weak solution of a deterministic radiative transfer equation. Let us note that in this case all the limit processes are deterministic. This fact means that the convergence also holds in probability. In the next section, we will see that this transport equation is well posed. In particular, this will imply the convergence of the process $(W_\e)_\e$ himself to the unique solution of the radiative transfer equation.

\begin{prop}\label{identification}
Let 	$W$ be a accumulation point of $(W_\e)_\e$. Then, $W$ is a weak solution of the radiative transfer equation
 \[
\partial_t W +\xt{k}\cdot \nabla_{\xt{x}} W=\mathcal{L}W,
\]
with $W(0,\xt{x},\xt{k})=W_0(\xt{x},\xt{k})$.
Here, $\mathcal{L}$ is defined by
\[
\mathcal{L}\varphi(\xt{k})=\int d\xt{p}\sigma(\xt{p}-\xt{k})\big(\varphi(\xt{p})-\varphi(\xt{k})\big)=\int d\xt{p}\sigma(\xt{p})\big(\varphi(\xt{k}+\xt{p})-\varphi(\xt{k})\big)\quad \forall \varphi \in \mathcal{C}^\infty_0(\mathbb{R}^d),
\]
with
\[\sigma(\xt{p})=\frac{2\widehat{R}_0(\xt{p})}{(2\pi)^d\mathfrak{g}(\xt{p})}.\] 
\end{prop}
\begin{proof}[of Proposition \ref{identification}]
In this proof we use the following notation
\[\varphi\otimes\psi(\xt{x}_1,\xt{k}_1,\xt{x}_2,\xt{k}_2)= \varphi(\xt{x}_1,\xt{k}_1)\psi(\xt{x}_2,\xt{k}_2).\]
 Let
\[
\begin{split}
f^\e _2 (t)=&\partial_v f(W_{\e,\lambda} (t))\big<W_\e(t),H_{1,\e}(t)\big>_{L^2(\mathbb{R}^{2d})}\\
&+\partial^2_v f(W_{\e,\lambda} (t))\big<W_\e(t)\otimes W_\e(t),H_{2,\e}(t)\big>_{L^2(\mathbb{R}^{4d})},
\end{split}
\]
where
\[\begin{split}
H_{1,\e}(t,&\xt{x},\xt{k})=\\
&\frac{1}{(2\pi)^{2d}i^2}\int_{t}^{+\infty}du\iint \Big(\mathbb{E}^\e _t\big[ \widehat{V}(u/\e^{1+\ga},d\xt{p}_1) \widehat{V}(u/\e^{1+\ga},d\xt{p}_2)\big] -\mathbb{E}\big[ \widehat{V}(0,d\xt{p}_1) \widehat{V}(0,d\xt{p}_2)\big] \Big) \\
&\times  e^{i(\xt{p}_1+\xt{p}_2)\cdot\xt{x}/\e}e^{i(u-t)(\xt{p}_1+\xt{p}_2)\cdot \xt{k}/\e}\\
&\times\Big[\frac{1}{\mathfrak{g}(\xt{\xt{p}_2})-i \e^\ga \big(\xt{k}-\frac{\xt{p}_1}{2}\big)\cdot\xt{p}_2}\Big(\lambda\big(\xt{x},\xt{k}-\frac{\xt{p}_1}{2}-\frac{\xt{p}_2}{2}\big)-\lambda\big(\xt{x},\xt{k}-\frac{\xt{p}_1}{2}+\frac{\xt{p}_2}{2}\big)\Big)\\
&\hspace{2cm}-\frac{1}{\mathfrak{g}(\xt{\xt{p}_2})-i\e^\ga\big(\xt{k}+\frac{\xt{p}_1}{2}\big)\cdot\xt{p}_2}\Big(\lambda\big(\xt{x},\xt{k}+\frac{\xt{p}_1}{2}-\frac{\xt{p}_2}{2}\big)-\lambda\big(\xt{x},\xt{k}+\frac{\xt{p}_1}{2}+\frac{\xt{p}_2}{2}\big)\Big)\Big] ,
\end{split}\]
and
\[\begin{split}
H_{2,\e}(t,&\xt{x}_1,\xt{k}_1,\xt{x}_2,\xt{k}_2)=\\
&\frac{1}{(2\pi)^{2d}i^2}\int_{t}^{+\infty}du\iint \Big(\mathbb{E}^\e _t\big[ \widehat{V}(u/\e^{1+\ga},d\xt{p}_1) \widehat{V}(u/\e^{1+\ga},d\xt{p}_2)\big] -\mathbb{E}\big[ \widehat{V}(0,d\xt{p}_1) \widehat{V}(0,d\xt{p}_2)\big] \Big) \\
&\times e^{i\xt{p}_1\cdot\xt{x}_1/\e}e^{i\xt{p}_2\cdot\xt{x}_2/\e}  e^{i(u-t)(\xt{p}_1\cdot \xt{k}_1+\xt{p}_2\cdot \xt{k}_2)/\e}\frac{1}{\mathfrak{g}(\xt{\xt{p}_1})-i\e^\ga\xt{k}_1\cdot\xt{p}_1}\\
&\times \Big(\lambda\big(\xt{x}_1,\xt{k}_1-\frac{\xt{p}_1}{2}\big)-\lambda\big(\xt{x}_1,\xt{k}_1+\frac{\xt{p}_1}{2}\big)\Big)\\
&\times\Big(\lambda\big(\xt{x}_2,\xt{k}_2-\frac{\xt{p}_2}{2}\big)-\lambda\big(\xt{x}_2,\xt{k}_2+\frac{\xt{p}_2}{2}\big)\Big).
\end{split}\]
However, according to \eqref{markovvar}
\[\begin{split}
\E^\e_t\Big[\widehat{V}\Big(u+&\frac{t}{\e^{1+\ga}},d\xt{p}_1\Big)\widehat{V}\Big(u+\frac{t}{\e^{1+\ga}},d\xt{p}_2\Big)\Big]-\E\Big[\widehat{V}(0,d\xt{p}_1)\widehat{V}(0,d\xt{p}_2)\Big]\\
&=e^{-(\mathfrak{g}(\xt{p}_1)+\mathfrak{g}(\xt{p}_2))u}\widehat{V}\Big(\frac{t}{\e^{1+\ga}},d\xt{p}_1\Big)\widehat{V}\Big(\frac{t}{\e^{1+\ga}},d\xt{p}_2\Big)-(2\pi)^d e^{-2\mathfrak{g}(\xt{p}_1)u}\widehat{R}_0(\xt{p}_1)\delta(\xt{p}_1+\xt{p}_2),
\end{split}\]
and
\[\begin{split}
\E^\e_t\Big[\widehat{V}\Big(u+&\frac{t}{\e^{1+\ga}},d\xt{p}_1\Big)\widehat{V}^\ast\Big(u+\frac{t}{\e^{1+\ga}},d\xt{p}_2\Big)\Big]-\E\Big[\widehat{V}(0,d\xt{p}_1)\widehat{V}^\ast(0,d\xt{p}_2)\Big]\\
&=e^{-(\mathfrak{g}(\xt{p}_1)+\mathfrak{g}(\xt{p}_2))u}\widehat{V}\Big(\frac{t}{\e^{1+\ga}},d\xt{p}_1\Big)\widehat{V}^\ast\Big(\frac{t}{\e^{1+\ga}},d\xt{p}_2\Big)-(2\pi)^d e^{-2\mathfrak{g}(\xt{p}_1)u}\widehat{R}_0(\xt{p}_1)\delta(\xt{p}_1-\xt{p}_2).
\end{split}\]
Consequently,
\[
\begin{split}
f^\e _2 (t)=&\e^{1+\ga}\Big[\partial_v f(W_{\e,\lambda} (t))\big<W_\e(t),\tilde{H}_{1,\e}(t)\big>_{L^2(\mathbb{R}^{2d})}\\
&+\partial^2_v f(W_{\e,\lambda} (t))\big<W_\e(t)\otimes W_\e(t),\tilde{H}_{2,\e}(t)\big>_{L^2(\mathbb{R}^{4d})}\Big],
\end{split}
\]
where
\[\begin{split}
\tilde{H}_{1,\e}(t,\xt{x},&\xt{k})=\frac{1}{(2\pi)^{2d}i^2}\iint \widehat{V}(t/\e^{1+\ga},d\xt{p}_1) \widehat{V}(t/\e^{1+\ga},d\xt{p}_2)\frac{ e^{i(\xt{p}_1+\xt{p}_2)\cdot\xt{x}/\e}}{(\mathfrak{g}(\xt{p}_1)+\mathfrak{g}(\xt{p}_2)-i\e^\ga\xt{k}\cdot(\xt{p}_1+\xt{p}_2))}\\
&\times  \Big[\frac{1}{\mathfrak{g}(\xt{\xt{p}_2})-i\e^\ga\big(\xt{k}-\frac{\xt{p}_1}{2}\big)\cdot\xt{p}_2}\Big(\lambda\big(\xt{x},\xt{k}-\frac{\xt{p}_1}{2}-\frac{\xt{p}_2}{2}\big)-\lambda\big(\xt{x},\xt{k}-\frac{\xt{p}_1}{2}+\frac{\xt{p}_2}{2}\big)\Big)\\
&\hspace{1cm}-\frac{1}{\mathfrak{g}(\xt{\xt{p}_2})-i\e^\ga\big(\xt{k}+\frac{\xt{p}_1}{2}\big)\cdot\xt{p}_2}\Big(\lambda\big(\xt{x},\xt{k}+\frac{\xt{p}_1}{2}-\frac{\xt{p}_2}{2}\big)-\lambda\big(\xt{x},\xt{k}+\frac{\xt{p}_1}{2}+\frac{\xt{p}_2}{2}\big)\Big)\Big]\\
&+\frac{1}{(2\pi)^d}\int d\xt{p}\frac{\widehat{R}_0(\xt{p})}{2\mathfrak{g}(\xt{p})}\Big[\frac{1}{\mathfrak{g}(\xt{\xt{p}})+i\e^\ga \big(\xt{k}-\frac{\xt{p}}{2}\big)\cdot\xt{p}}\Big(\lambda(\xt{x},\xt{k})-\lambda(\xt{x},\xt{k}-\xt{p})\Big)\\
&\hspace{1cm}-\frac{1}{\mathfrak{g}(\xt{\xt{p}})+i\e^\ga\big(\xt{k}+\frac{\xt{p}}{2}\big)\cdot\xt{p}}\Big(\lambda(\xt{x},\xt{k}+\xt{p})-\lambda(\xt{x},\xt{k})\Big)\Big],
\end{split}\]
and
\[\begin{split}
\tilde{H}_{2,\e}(t,\xt{x}_1,\xt{k}_1,\xt{x}_2,\xt{k}_2)=&\frac{1}{(2\pi)^{2d}i^2} \iint \widehat{V}(t/\e^{1+\ga},d\xt{p}_1) \widehat{V}(t/\e^{1+\ga},d\xt{p}_2) \\
&\times \frac{e^{i\xt{p}_1\cdot\xt{x}_1/\e}e^{i\xt{p}_2\cdot\xt{x}_2/\e}}{(\mathfrak{g}(\xt{\xt{p}_1})-i\e^\ga\xt{k}_1\cdot\xt{p}_1)(\mathfrak{g}(\xt{\xt{p}_1})+\mathfrak{g}(\xt{\xt{p}_2})-i\e^\ga(\xt{k}_1\cdot\xt{p}_1+\xt{k}_2\cdot\xt{p}_2))}\\
&\times \Big(\lambda\big(\xt{x}_1,\xt{k}_1-\frac{\xt{p}_1}{2}\big)-\lambda\big(\xt{x}_1,\xt{k}_1+\frac{\xt{p}_1}{2}\big)\Big)\\
&\times\Big(\lambda\big(\xt{x}_2,\xt{k}_2-\frac{\xt{p}_2}{2}\big)-\lambda\big(\xt{x}_2,\xt{k}_2+\frac{\xt{p}_2}{2}\big)\Big)\\
&+\frac{1}{(2\pi)^d}\int d\xt{p} \frac{\widehat{R}_0(\xt{p})e^{i\xt{p}\cdot(\xt{x}_1-\xt{x}_2)/\e}}{(\mathfrak{g}(\xt{p})-i\e^\ga\xt{k}_1\cdot\xt{p})(2\mathfrak{g}(\xt{p})-i\e^\ga(\xt{k}_1-\xt{k}_2)\cdot\xt{p} )}\\
&\times \Big(\lambda\big(\xt{x}_1,\xt{k}_1-\frac{\xt{p}}{2}\big)-\lambda\big(\xt{x}_1,\xt{k}_1+\frac{\xt{p}}{2}\big)\Big)\\
&\times\Big(\lambda\big(\xt{x}_2,\xt{k}_2+\frac{\xt{p}}{2}\big)-\lambda\big(\xt{x}_2,\xt{k}_2-\frac{\xt{p}}{2}\big)\Big).
\end{split}\]

\begin{lem}\label{bound6}
\[ \sup_{t\geq 0} \E[\lvert f^\e_2(t)\rvert]=\mathcal{O}(\e^{1+\ga}).\]
\end{lem}
\begin{proof}[of Lemma \ref{bound6}]
In the same way as in Lemma \ref{bound4}, using \eqref{longrange}, \eqref{vargaussien}, the smoothness of $\lambda$, and that $\lambda$ and $\widehat{R}_0$ have a support included in a compact set, we have
\[\begin{split}
\sup_{\e,t}&\E[\|\tilde{H}_{1,\e}(t)\|^2_{L^{\mathbb{R}^{2d}}}]
\leq C \iint d\xt{p}_1 d\xt{p}_2\widehat{R}_0(\xt{p}_1)\widehat{R}_0(\xt{p}_2)\\
&\times\Big[\Big(\frac{\lvert \xt{p}_1\rvert^3}{\mathfrak{g}(\xt{p}_1)^3} \sup_{\xt{x},\xt{k}}\lvert \nabla_{\xt{k}}\lambda(\xt{x},\xt{k})\rvert+\frac{\lvert\xt{p}_1\rvert^2 }{\mathfrak{g}^2(\xt{p}_1)}\sup_{\xt{x},\xt{k}}\|D^2\lambda(\xt{x},\xt{k})\|  \Big) \\
&\hspace{0.5cm}\times\Big(\frac{\lvert \xt{p}_2\rvert^3}{\mathfrak{g}^3(\xt{p}_2)} \sup_{\xt{x},\xt{k}}\lvert \nabla_{\xt{k}}\lambda(\xt{x},\xt{k})\rvert+\frac{\lvert\xt{p}_2\rvert^2 }{\mathfrak{g}^2(\xt{p}_2)}\sup_{\xt{x},\xt{k}}\|D^2\lambda(\xt{x},\xt{k})\|  \Big) \\
&+\Big(\frac{\lvert \xt{p}_1\rvert\lvert \xt{p}_2\rvert^2}{\mathfrak{g}(\xt{p}_1)\mathfrak{g}^2(\xt{p}_2)}\sup_{\xt{x},\xt{k}}\lvert \nabla_{\xt{x}}\lambda(\xt{x},\xt{k})\rvert+\Big(\frac{\lvert \xt{p}_1\rvert\lvert \xt{p}_2\rvert}{\mathfrak{g}(\xt{p}_1)\mathfrak{g}(\xt{p}_2)}+\frac{\lvert \xt{p}_2\rvert^2}{\mathfrak{g}^2(\xt{p}_2)}\Big)\sup_{\xt{x},\xt{k}}\|D^2\lambda(\xt{x},\xt{k})\|\Big)^2\\
&+\Big(\frac{\lvert \xt{p}_1\rvert\lvert \xt{p}_2\rvert^2}{\mathfrak{g}(\xt{p}_1)\mathfrak{g}^2(\xt{p}_2)}\sup_{\xt{x},\xt{k}}\lvert \nabla_{\xt{x}}\lambda(\xt{x},\xt{k})\rvert+\Big(\frac{\lvert \xt{p}_1\rvert\lvert \xt{p}_2\rvert}{\mathfrak{g}(\xt{p}_1)\mathfrak{g}(\xt{p}_2)}+\frac{\lvert \xt{p}_2\rvert^2}{\mathfrak{g}^2(\xt{p}_2)}\Big)\sup_{\xt{x},\xt{k}}\|D^2\lambda(\xt{x},\xt{k})\|\Big)\\
&\hspace{0.5cm}\times\Big(\frac{\lvert \xt{p}_2\rvert\lvert \xt{p}_1\rvert^2}{\mathfrak{g}(\xt{p}_2)\mathfrak{g}^2(\xt{p}_1)}\sup_{\xt{x},\xt{k}}\lvert \nabla_{\xt{x}}\lambda(\xt{x},\xt{k})\rvert+\Big(\frac{\lvert \xt{p}_2\rvert\lvert \xt{p}_1\rvert}{\mathfrak{g}(\xt{p}_2)\mathfrak{g}(\xt{p}_1)}+\frac{\lvert \xt{p}_1\rvert^2}{\mathfrak{g}^2(\xt{p}_1)}\Big)\sup_{\xt{x},\xt{k}}\|D^2\lambda(\xt{x},\xt{k})\|\Big)\Big]<+\infty,
\end{split}\]
and
\[
\sup_{\e,t}\E[\|\tilde{H}_{2,\e}(t)\|^2_{L^{\mathbb{R}^{2d}}}]\leq C \iint d\xt{p}_1 d\xt{p}_2\widehat{R}_0(\xt{p}_1)\widehat{R}_0(\xt{p}_2)\frac{\lvert \xt{p}_1\rvert^2\lvert \xt{p}_2\rvert^2}{\mathfrak{g}^2(\xt{p}_1)\mathfrak{g}^2(\xt{p}_2)}\sup_{\xt{x},\xt{k}}\lvert \nabla_{\xt{k}}\lambda(\xt{x},\xt{k})\rvert^4.
\]
That concludes the proof of Lemma \ref{bound6}.
$\square$
\end{proof}
Consequently, after a long but straightforward computation, we get
\begin{equation}\label{A2}\begin{split}
\mathcal{A}^\e(f^\e _0 +f^\e _1+f^\e_2)(t)=&\partial_v f(W_{\e,\lambda}(t))\big[W_{\e,\lambda_1}(t)+\big<W_\e(t),G_{1,\e}\lambda\big>_{L^2(\mathbb{R}^{2d})}\big]\\
&+\partial^2_v f(W_{\e,\lambda}(t))\big<W_\e(t)\otimes W_\e(t),G_{2,\e}\lambda \big>_{L^2(\mathbb{R}^{4d})}\\
&+\mathcal{O}(\e^{(1+\ga)/2}),
\end{split}\end{equation}
where
\[\begin{split}
G_{1,\e}\lambda(\xt{x},\xt{k})=&-\frac{1}{(2\pi)^d}\int d\xt{p}\widehat{R}_0(\xt{p})\Big[\frac{1}{\mathfrak{g}(\xt{\xt{p}})-i\e^\ga\big(\xt{k}-\frac{\xt{p}}{2}\big)\cdot\xt{p}}\Big(\lambda(\xt{x},\xt{k})-\lambda(\xt{x},\xt{k}-\xt{p})\Big)\\
&\hspace{2cm}-\frac{1}{\mathfrak{g}(\xt{\xt{p}})-i\e^\ga\big(\xt{k}+\frac{\xt{p}}{2}\big)\cdot\xt{p}}\Big(\lambda(\xt{x},\xt{k}+\xt{p})-\lambda(\xt{x},\xt{k})\Big)\Big],
\end{split}\]

\[\begin{split}
G_{2,\e}\lambda(\xt{x}_1,\xt{k}_1,\xt{x}_2,\xt{k}_2)=&-\frac{1}{(2\pi)^d}\int d\xt{p} \frac{\mathfrak{g}(\xt{p})\widehat{R}_0(\xt{p})e^{i\xt{p}\cdot(\xt{x}_1-\xt{x}_2)/\e}}{(\mathfrak{g}(\xt{p})-i\e^\ga\xt{k}_1\cdot\xt{p})(2\mathfrak{g}(\xt{p})-i\e^\ga(\xt{k}_1-\xt{k}_2)\cdot\xt{p} )}\\
&\times \Big(\lambda\big(\xt{x}_1,\xt{k}_1-\frac{\xt{p}}{2}\big)-\lambda\big(\xt{x}_1,\xt{k}_1+\frac{\xt{p}}{2}\big)\Big)\\
&\times\Big(\lambda\big(\xt{x}_2,\xt{k}_2+\frac{\xt{p}}{2}\big)-\lambda\big(\xt{x}_2,\xt{k}_2-\frac{\xt{p}}{2}\big)\Big),
\end{split}\]
and
\[\begin{split}
G_{3,\e}\lambda(\xt{x}_1,\xt{k}_1,\xt{x}_2,\xt{k}_2)=&-\frac{1}{(2\pi)^d}\int d\xt{p}\frac{\mathfrak{g}(\xt{p})\widehat{R}_0(\xt{p})e^{i\xt{p}\cdot(\xt{x}_1-\xt{x}_2)/\e}}{(\mathfrak{g}(\xt{p})-i\e^\ga\xt{k}_1\cdot\xt{p})(2\mathfrak{g}(\xt{p})-i\e^\ga(\xt{k}_1-\xt{k}_2)\cdot\xt{p} )}\\
&\times \Big(\lambda\big(\xt{x}_1,\xt{k}_1-\frac{\xt{p}}{2}\big)-\lambda\big(\xt{x}_1,\xt{k}_1+\frac{\xt{p}}{2}\big)\Big)\\
&\times\overline{\Big(\lambda\big(\xt{x}_2,\xt{k}_2+\frac{\xt{p}}{2}\big)-\lambda\big(\xt{x}_2,\xt{k}_2-\frac{\xt{p}}{2}\big)\Big)}.
\end{split}\]

\begin{lem}\label{bound7}
\[\lim_{\e} \|  G_{2,\e}\lambda\|^2 _{L^2(\mathbb{R}^{4d})}+ \|  G_{3,\e}\lambda\|^2 _{L^2(\mathbb{R}^{4d})}=0.\]
\end{lem}
\begin{proof}[of Lemma \ref{bound7}]
We will just show the result for $G_{2,\e}$, the proof is exactly the same for $G_{3,\e}$.
\[\|G_{2,\e}\lambda\|^2_{L^2(\mathbb{R}^{4d})}=\int d\xt{x}_1 d\xt{x}_2 d\xt{k}_1 d\xt{k}_2\Big\vert \int d\xt{p} \Psi_\e(\xt{x}_1,\xt{x}_2,\xt{k}_1,\xt{k}_2,\xt{p}) e^{i \xt{p}\cdot (\xt{x}_1-\xt{x}_2)/\e}\Big\vert^2
 \]
where 
\[\begin{split}
\Psi_\e(\xt{x}_1,\xt{x}_2,\xt{k}_1,\xt{k}_2,\xt{p})=&\frac{\mathfrak{g}(\xt{p})\widehat{R}_0(\xt{p})}{(\mathfrak{g}(\xt{p})-i\e^\ga\xt{k}_1\cdot \xt{p})(2\mathfrak{g}(\xt{p})-i\e^\ga(\xt{k}_1-\xt{k}_2)\cdot \xt{p})} \\
&\times\xt{p}\cdot \int_{-1/2}^{1/2} du_1\nabla_{\xt{k}}\lambda (\xt{x}_1,\xt{k}_1+u_1 \xt{p})\times \xt{p}\cdot \int_{-1/2}^{1/2}du_2 \nabla_{\xt{k}}\lambda (\xt{x}_2, \xt{k}_2+u_2 \xt{p}).
\end{split}\]
Let us recall that $\widehat{R}_0$ and $\lambda$ have their supports included in a compact set. 
Then,
\[\sup_{\xt{x}_1,\xt{k}_1,\xt{x}_2,\xt{k}_2}\lvert\Psi_\e(\xt{x}_1,\xt{x}_2,\xt{k}_1,\xt{k}_2,\xt{p})- \Psi_0(\xt{x}_1,\xt{x}_2,\xt{k}_1,\xt{k}_2,\xt{p})\rvert\leq \e^\ga C\frac{\widehat{R}_0(\xt{p})\lvert p\rvert^2}{\mathfrak{g}(\xt{p})}\]
with
\[\Psi_0(\xt{x}_1,\xt{x}_2,\xt{k}_1,\xt{k}_2,\xt{p})=\frac{\widehat{R}_0(\xt{p})}{2\mathfrak{g}(\xt{p})}\xt{p}\cdot \int_{-1/2}^{1/2} du_1\nabla_{\xt{k}}\lambda (\xt{x}_1,\xt{k}_1+u_1 \xt{p})\times \xt{p}\cdot \int_{-1/2}^{1/2}du_2 \nabla_{\xt{k}}\lambda (\xt{x}_2, \xt{k}_2+u_2 \xt{p}).\]
Moreover, for almost every $(\xt{x}_1,\xt{x}_2,\xt{k}_1,\xt{k}_2)$
\[ \lim_\e \int d\xt{p} \Psi(\xt{x}_1,\xt{x}_2,\xt{k}_1,\xt{k}_2,\xt{p}) e^{i \xt{p}\cdot (\xt{x}_1-\xt{x}_2)/\e}=0\]
by the Riemann--Lebesgue lemma, since $\widehat{R}(\xt{p})\frac{\lvert \xt{p}\rvert^2 }{\mathfrak{g}(\xt{p})}\in L^1(\mathbb{R}^d)$.
Consequently, the result of Lemma \ref{bound7} follows from the dominated convergence theorem.$\square$
\end{proof}

\begin{lem}\label{bound8}
We have
\[\lim_{\e} \|  (G_{1,\e}-G_1)\lambda\|^2 _{L^2(\mathbb{R}^{2d})}=0,\]
with 
\[\begin{split}
G_{1}\lambda(\xt{x},\xt{k})&=\frac{1}{(2\pi)^d}\int d\xt{p}\frac{\widehat{R}_0(\xt{p})}{\mathfrak{g}(\xt{p})}\Big(\lambda(\xt{x},\xt{k}+\xt{p})+\lambda(\xt{x},\xt{k}-\xt{p})-2\lambda(\xt{x},\xt{k})\Big)\\
&=\frac{1}{(2\pi)^d}\int d\xt{p}\frac{2\widehat{R}_0(\xt{p})}{\mathfrak{g}(\xt{p})}\Big(\lambda(\xt{x},\xt{k}+\xt{p})-\lambda(\xt{x},\xt{k})\Big).
\end{split}\]
\end{lem}
\begin{proof}[of Lemma \ref{bound8}]
This Lemma is a direct consequence of $\widehat{R}(\xt{p})\frac{\lvert \xt{p}\rvert^2 }{\mathfrak{g}^2(\xt{p})}\in L^1(\mathbb{R}^d)$.
$\square$
\end{proof}

Let $f^\e (t)=f^\e _0 (t) + f^\e _1 (t)+ f^\e _2 (t)$. By Theorem \ref{martingale}, $\big(M^\e _{f^\e} (t) \big)_{t\geq 0}$ is an $\left( \mathcal{F}^\e _t \right)$-martingale. That is, for every bounded continuous function $\Phi$, every sequence $0< s_1<\cdots <s_n \leq s <t$, and every family $(\mu_j)_{j\in \{1,\dots,n\}}\in L^2(\mathbb{R}^d)^n$, we have
\[\mathbb{E}\Big[ \Phi\big(W _{\e,\mu_j}(s_j),1\leq j \leq n\big)\Big( f^\e (t) - f^\e (s)-\int _s ^t  \mathcal{A}^\e f^\e (u)du \Big) \Big]=0 .\]
Using \eqref{A2}, Lemma \ref{bound2}, Lemma \ref{bound6}, Lemma \ref{bound7}, and Lemma \ref{bound8} we obtain that
\[\begin{split}
M_{f,\lambda}(t)=&f (W_\lambda(t)) - f(W_\lambda(0))\\
&-\int _0^t  \partial_v f(W_{\lambda}(u))\big[W_{\lambda_1}(u)+\big<W(u),G_{1}\lambda\big>_{L^2(\mathbb{R}^{2d})}\big]
\end{split}\]
is a martingale, and where $\lambda_1$ is defined by \eqref{lambda1}. More particularly, Let us consider $f$ be a smooth function so that $f(v)=v$, $\forall v$ such that $\lvert v\rvert\leq r\|\lambda\|_{L^2(\mathbb{R}^{2d})}$. Then,
\[\begin{split}
M_{\lambda}(t)=&W_\lambda(t) - W_\lambda(0)\\
&-\int _0 ^t  du \big[W_{\lambda_1}(u)+\big<W(u),G_{1}\lambda\big>_{L^2(\mathbb{R}^{2d})}\big]
\end{split}\]
is a martingale with a quadratic variation equal to $0$. Consequently, $M_{\lambda}=0$, that concludes the proof a Proposition \ref{identification}. $\blacksquare$
\end{proof}

\subsection{Proof of the weak uniqueness of the radiative transfer equation \eqref{radtranseq}}

In this section we show the uniqueness of weak solutions of the radiative transfer equation \eqref{radtranseq}. To this end will construct a good test function using an approximation $\mathcal{L}_N$ defined 
\[\mathcal{L}_N\varphi(\xt{k})=\int_{\lvert\xt{p}\rvert>1/N } d\xt{p}\sigma(\xt{p})\big(\varphi(\xt{p}+\xt{k})-\varphi(\xt{k})\big).\]
of $\mathcal{L}$ defined by
\[\mathcal{L}\varphi(\xt{k})=\int d\xt{p}\sigma(\xt{p})\big(\varphi(\xt{p}+\xt{k})-\varphi(\xt{k})\big),\]
$\forall \varphi\in\mathcal{C}^{\infty}_0(\mathbb{R}^d)$. $\mathcal{L}_N$ is an approximation of $\mathcal{L}$ with a finite scattering coefficient $\int_{\lvert\xt{p}\rvert>1/N } d\xt{p}\sigma(\xt{p})$.  As we will see $\mathcal{L}_N$ is a well suited approximation of $\mathcal{L}$ to construct a test function which enables us to show the weak uniqueness of the transfer equation.

\begin{prop}\label{testfunction}
Let $\lambda\in\mathcal{C}^{\infty}_0(\mathbb{R}^{2d})$. $\forall N\geq1$, there exists a unique solution $W_N$ of the radiative transfer equation 
\begin{equation}\label{approxtransport}
\partial_t W_N(t,\xt{x},\xt{k})=\xt{k}\cdot\nabla_{\xt{x}} W_N(t,\xt{x},\xt{k})+\mathcal{L}_NW_N(t,\xt{x},\xt{k})\end{equation}
with $W_N(0,\xt{x},\xt{k})=\lambda(\xt{x},\xt{k})$, such that $W_N\in \mathcal{C}^1([0,+\infty),L^2(\mathbb{R}^{2d}))$. $W_N$ is given by the series expansion   
\begin{equation}\label{seriesexp}\begin{split}W_N(t,\xt{x},\xt{k})=&e^{-\Sigma_N t}\lambda(\xt{x}+t\xt{k},\xt{k})\\
&+e^{-\Sigma_N t}\sum_{n\geq 1}\int_{D_n(t)}d\xt{s}^{(n)} \int d\xt{p}^{(n)} \prod_{j=1}^n \sigma(\xt{p}_j)\lambda\Big(\xt{x}+t\xt{k}+\sum_{j=1}^n s_j\xt{p}_j,\xt{k}+\sum_{j=1}^n\xt{p}_j\Big).\end{split}\end{equation}
Here, $D_n(t)=\{\xt{s}^{(n)}=(s_1,\dots,s_n):\,\,0\leq s_n\leq\dots\leq s_1\leq t\}$, $\xt{p}^{(n)}=(\xt{p}_1,\dots,\xt{p}_n)$, and
\[\Sigma_N=\int_{\lvert\xt{p}\rvert>1/N } d\xt{p}\sigma(\xt{p}).\]
 \end{prop}

Let $W$ be a weak solution of the radiative transfer equation  \eqref{radtranseq} with a null initial condition. Using the test function $W_N$ given by Proposition \ref{testfunction}, let us consider $\tilde{W}_N(t,\xt{x},\xt{k})=W_N(T-t,\xt{x},\xt{k})$, for $T>0$. Then, $\tilde{W}_N$ satisfies 
\[\partial_t \tilde{W}_N(t,\xt{x},\xt{k})+\xt{k}\cdot\nabla_{\xt{x}} \tilde{W}_N(t,\xt{x},\xt{k})+\mathcal{L}_N\tilde{W}_N(t,\xt{x},\xt{k})=0\]
with $\tilde{W}_N(T,\xt{x},\xt{k})=\lambda(\xt{x},\xt{k})$. Consequently, we have following lemma.
\begin{lem}\label{extension} $\forall t\in(0,T)$, we have
\[\begin{split}
\frac{d}{dt}\big<W(t),\tilde{W}_N(t)\big>_{L^2(\mathbb{R}^{2d})}&=\big<W(t),\partial_t \tilde{W}_N(t) +\xt{k}\cdot\nabla_{\xt{x}}\tilde{W}_N(t)+\mathcal{L}\tilde{W}_N(t)\big>_{L^2(\mathbb{R}^{2d})}\\
&=\big<W(t),(\mathcal{L}-\mathcal{L}_N)\tilde{W}_N(t)\big>_{L^2(\mathbb{R}^{2d})}.\end{split}\]
\end{lem}
As a result,
\[\big<W(T),\lambda\big>_{L^2(\mathbb{R}^{2d})}=\big<W(T),\tilde{W}_N(T)\big>_{L^2(\mathbb{R}^{2d})}=\int_0^Tds\big<W(s),(\mathcal{L}-\mathcal{L}_N)\tilde{W}_N(s)\big>_{L^2(\mathbb{R}^{2d})}.\]
Finally, it remains to show the following lemma.
\begin{lem}\label{approxgene}
\[\sup_{t\in[0,T]}\|(\mathcal{L}-\mathcal{L}_N)\tilde{W}_N(t)\|_{L^2(\mathbb{R}^{2d})}\leq C_{\lambda}\int_{\lvert \xt{p}\rvert <1/N}\frac{\widehat{R}_0(\xt{p})\lvert \xt{p}\rvert }{\mathfrak{g}(\xt{p})}. \]
\end{lem}

This last lemma gives us the weak uniqueness. In fact, we have  $\int d\xt{p} \widehat{R}_0(\xt{p})\lvert \xt{p}\rvert /\mathfrak{g}(\xt{p})<+\infty$, and then,
\[\lim_{N\to+\infty }\int_{\lvert \xt{p}\rvert <1/N}d\xt{p}\frac{\widehat{R}_0(\xt{p})\lvert \xt{p}\rvert }{\mathfrak{g}(\xt{p})}=0.\]
Consequently, $\forall\lambda\in\mathcal{C}^{\infty}_0(\mathbb{R}^{2d})$ and $\forall T>0$, $\big<W(T),\lambda\big>_{L^2(\mathbb{R}^{2d})}=0$.

\begin{proof}[of Propostion \ref{testfunction}]
Let us rewrite \eqref{approxtransport}
using the finite scattering coefficient $\Sigma_N$. Then, We are looking for a solution of 
\[\partial_t W_N(t,\xt{x},\xt{k})=\xt{k}\cdot\nabla_{\xt{x}} W_N(t,\xt{x},\xt{k})-\Sigma_NW_N(t,\xt{x},\xt{k})+\int d\xt{p} \sigma(\xt{p})W_N(t,\xt{x},\xt{k}+\xt{p})\]
This equation can be recast by integrating as follows
\[W_N(t,\xt{x},\xt{k})=e^{-\Sigma_N t}\lambda(\xt{x}+t\xt{k},\xt{k})+\int_0^t ds \int_{\lvert \xt{p}\rvert>1/N} d\xt{p} \sigma(\xt{p})e^{-(t-s)\Sigma_N}W_N(s,\xt{x}+(t-s)\xt{k},\xt{k}+\xt{p}),\]
and then the series expansion \eqref{seriesexp} is obtained by induction. Let us write 
\[W_N(t,\xt{x},\xt{k})=\sum_{n\geq0} W^n_N(t,\xt{x},\xt{k})\]
where $W^n_N$ are the terms of the serie expansion \eqref{seriesexp}. Using changes of variables we show that $\forall t\geq 0$, $W_N(t)\in L^2(\mathbb{R}^d)$ and
\[\sum_{n\geq 0} \sup_{t}\| W^n_N(t)\|_{L^2(\mathbb{R}^{2d})}<+\infty.\]
Then $W_N\in \mathcal{C}^0([0,+\infty),L^2(\mathbb{R}^{2d}))$. Moreover, 
\begin{equation}\label{sumgrad}\sum_{n\geq 0} \sup_{0\leq t\leq T}\| \xt{k}\cdot\nabla_{\xt{x}}W^n_N(t)\|_{L^2(\mathbb{R}^{2d})}\leq \|\xt{k}\cdot\nabla_{\xt{x}}\lambda\|_{L^2(\mathbb{R}^{2d})}+CT\Sigma_N \|\nabla_{\xt{x}}\lambda\|_{L^2(\mathbb{R}^{2d})} <+\infty,\end{equation}
since $\widehat{R}_0$ is assumed to be with compact support, and
\begin{equation}\label{sumgrad2}\sum_{n\geq 0} \sup_{t}\| \int d\xt{p}\sigma(\xt{p}) W^n_N(t,\cdot,\cdot+\xt{p})\|_{L^2(\mathbb{R}^{2d})}\leq\Sigma_N \|\lambda\|_{L^2(\mathbb{R}^{2d})} <+\infty.\end{equation}
$\blacksquare$
\end{proof}

\begin{proof}[of Lemma \ref{extension}]
First of all, let us note that equation \eqref{radtranseq} is valid for all $\lambda\in \mathcal{C}^{\infty}_0(\mathbb{R}^{2d})$. However, since $\sup_{t\geq0}\|W(t)\|_{L^2(\mathbb{R}^d)}<+\infty$ and $W\in C^0([0,+\infty),(\mathcal{B}_r,d_{\mathcal{B}_r}))$, we also have
\[
\big<W(t),\lambda(t)\big>_{L^2(\mathbb{R}^{2d})}=\big<W(0),\lambda(0)\big>_{L^2(\mathbb{R}^{2d})}+\int_0^t ds \big<W(s),\partial_t\lambda(s)+\xt{k}\cdot\nabla_{\xt{x}}\lambda(s)+\mathcal{L}\lambda(s)\big>_{L^2(\mathbb{R}^{2d})},
\] 
$\forall \lambda \in\mathcal{C}^1([0,T],\mathcal{C}^{\infty}_0(\mathbb{R}^{2d}))$, for which $\sup_t\|\lambda(t)\|_{L^2(\mathbb{R}^{2d})}+\sup_t\|\partial_t\lambda(t)\|_{L^2(\mathbb{R}^{2d})}<+\infty$.
Consequently, we have
\[\begin{split}
\big<W(t),\tilde{W}^n_N(t)\big>_{L^2(\mathbb{R}^{2d})}=&\big<W(0),\tilde{W}^n_N(0)\big>_{L^2(\mathbb{R}^{2d})}\\
&+\int_0^t ds \big<W(s),\partial_t\tilde{W}^n_N(s)+\xt{k}\cdot\nabla_{\xt{x}}\tilde{W}^n_N(s)+\mathcal{L}\tilde{W}^n_N(s)\big>_{L^2(\mathbb{R}^{2d})},
\end{split}\] 
where $\tilde{W}^n_N(t)=W^n_N(T-t)$. First, let us note that 
\[ W^n_N(t,\xt{x},\xt{k})=\int_0^t ds \int_{\lvert \xt{p}\rvert >1/N} d\xt{p}\sigma(\xt{p})e^{-(t-s)\Sigma_N}W^{n-1}_N(s,\xt{x}+(t-s)\xt{k},\xt{k}+\xt{p})  ,\]
so that, thanks to \eqref{sumgrad} and \eqref{sumgrad2}, we obtain
\[\sum_{n\geq 0} \sup_{0\leq t\leq T}\| \partial_t \tilde{W}^n_N(t)\|_{L^2(\mathbb{R}^{2d})}+\sum_{n\geq 0} \sup_{0\leq t\leq T}\| \xt{k}\cdot\nabla_{\xt{x}}\tilde{W}^n_N(t)\|_{L^2(\mathbb{R}^{2d})} <+\infty.\]
Moreover,
\[\mathcal{L}W^n_N(t,\xt{x},\xt{k})=\int d\xt{p} \sigma(\xt{p})\int_0^1 du \xt{p}\cdot\nabla_{\xt{k}}W^n_N(t,\xt{x},\xt{k}+u\xt{p}),\]
and then
\[ \sum_{n\geq1} \sup_{0\leq t\leq T} \|\mathcal{L}W^n_N(t)\|_{L^2(\mathbb{R}^{2d})}\leq \int d\xt{p} \frac{\widehat{R}_0(\xt{p})\lvert \xt{p}\rvert }{\mathfrak{g}(\xt{p})} \Big(T\|\nabla_{\xt{x}}\lambda\|_{L^2(\mathbb{R}^{2d})}+\|\nabla_{\xt{k}}\lambda\|_{L^2(\mathbb{R}^{2d})}\Big)<+\infty.\]
Consequently, we obtain the desire result
\[\begin{split}
\big<W(t),\tilde{W}_N(t)\big>_{L^2(\mathbb{R}^{2d})}=&\big<W(0),\tilde{W}_N(0)\big>_{L^2(\mathbb{R}^{2d})}\\
&+\int_0^t ds \big<W(s),\partial_t\tilde{W}_N(s)+\xt{k}\cdot\nabla_{\xt{x}}\tilde{W}_N(s)+\mathcal{L}\tilde{W}_N(t)(s)\big>_{L^2(\mathbb{R}^{2d})}\\
=&\big<W(0),\tilde{W}_N(0)\big>_{L^2(\mathbb{R}^{2d})}\\
&+\int_0^t ds \big<W(s),(\mathcal{L}-\mathcal{L}_N)\tilde{W}_N(s)\big>_{L^2(\mathbb{R}^{2d})}.
\end{split}\] 
$\square$
\end{proof}

\begin{proof}[of Lemma \ref{approxgene}]
We have
\[(\mathcal{L}-\mathcal{L}_N)\tilde{W}_N(t,\xt{x},\xt{k})=\int_{\lvert \xt{p}\rvert <1/N }d\xt{p}\sigma(\xt{p})\int_0^1du \xt{p}\cdot \nabla_{k}\tilde{W}_N(t,\xt{x},\xt{k}+u\xt{p}),\]
and then
\[\sup_{0\leq t\leq T}\|(\mathcal{L}-\mathcal{L}_N)\tilde{W}_N(t)\|_{L^2(\mathbb{R}^{2d})}\leq \int_{\lvert \xt{p}\rvert <1/N }d\xt{p}\frac{\widehat{R}_0(\xt{p})\lvert \xt{p}\rvert}{\mathfrak{g}(\xt{p})}\Big(T\|\nabla_{\xt{x}}\lambda\|_{L^2(\mathbb{R}^{2d})}+\|\nabla_{\xt{k}}\lambda\|_{L^2(\mathbb{R}^{2d})}\Big).\]
$\square$
\end{proof}

\section{Proof of Theorem \ref{regularity}}\label{proofthreg}

First, we let us compute the solution of \eqref{radtranseq}.

\begin{lem}\label{indsol}
$\forall W_0\in L^2(\mathbb{R}^{2d})$, the unique weak solution uniformly bounded in $ L^2(\mathbb{R}^{2d})$ of \eqref{radtranseq} is given by 
\begin{equation}\label{indsoleq}W(t,\xt{x},\xt{k})=\frac{1}{(2\pi)^{2d}}\int d\xt{y}d\xt{q} e^{i(\xt{x}\cdot\xt{y}+\xt{k}\cdot\xt{q})}\exp\Big(\int_0 ^t du \Psi(\xt{q}+u\xt{y})\Big)\widehat{W}_0(\xt{y},\xt{q}+t\xt{y}).\end{equation}
where $\Psi$ defined by \eqref{caracexp}.
\end{lem}

\begin{proof}[of Lemma \ref{indsol}]
First, let $W_0\in \mathcal{C}^\infty_0(\mathbb{R}^{2d})$ and let
\[\tilde{W}(t,\xt{y},\xt{q})=\exp\Big(\int_0 ^t du \Psi(\xt{q}+u\xt{y})\Big)\widehat{W}_0(\xt{y},\xt{q}+t\xt{y}).\]
$W$ is the inverse Fourier transform of $\tilde{W}$. Then,
\[\partial_t \tilde{W}=\xt{y}\cdot\nabla_\xt{q} \tilde{W}+\Psi(\xt{q})\tilde{W},\]
 and taking the inverse Fourier transform of this last equation we obtain that $W$ is a classical solution, and then a weak solution of \eqref{radtranseq} uniformly bounded in $L^2(\mathbb{R}^{2d})$. In fact, 
 \begin{equation}\label{unifboundreg}\|W(t)\|_{L^2(\mathbb{R}^{2d})}\leq \|W_0\|_{L^2(\mathbb{R}^{2d})}\end{equation}
 since
 \[\lvert \exp\Big(\int_0 ^t du \Psi(\xt{q}+u\xt{y})\Big)\Big \rvert=\exp\Big(-\int_0^t du \int d \xt{p}\sigma(\xt{p})(1-\cos(\xt{p}\cdot(\xt{q}+u\xt{y})))\Big)\leq 1. \]  
 Finally, using \eqref{unifboundreg}, the uniqueness of \eqref{radtranseq}, and by density, we obtain that  \eqref{indsoleq} is the unique solution of  \eqref{radtranseq}. That concludes the proof of Lemma \ref{indsol}.
$\square$
\end{proof}

Now, let us show the key lemma of the proof of Theorem \ref{regularity}. This lemma is a consequence of assumption \eqref{SDCAreg}. 

\begin{lem}\label{lemkey}
Let $\mu>0$ and $t\in(3/(2\mu),2/\mu)$. There exists $M_{t,\mu}>0$ and $C_{t,\theta,\mu}>0$, such that $\forall \lvert \xt{q}\rvert \geq M_{t,\mu}$ and $\forall \lvert \xt{y}\rvert \geq M_{t,\mu}$
\[
\exp\Big(\int_0^{t}du \Psi(\xt{q}+u\mu\xt{y})\Big) \leq \exp(-C_{t,\theta,\mu}\lvert \xt{q} \rvert^{\theta})\emph{\1}_{(\lvert \xt{q}\rvert \geq \mu t \lvert \xt{y}\vert)}+\exp(-C_{t,\theta,\mu}\lvert \xt{y} \rvert^{\theta})\emph{\1}_{(\lvert \xt{q}\rvert< \mu t \lvert \xt{y}\vert)}.
\]
\end{lem}
\begin{proof}[of Lemma \ref{lemkey}]

This lemma is a consequence of assumption \eqref{SDCAreg}.  First, we have
\[\begin{split}
\int_0^{t}du \Psi(\xt{q}+u\mu\xt{y}) &= - \int_0^{t}du \int d\xt{p} \sigma(\xt{p})(1-\cos(\xt{p}\cdot(\xt{q}+u\mu\xt{y})))\\
&\leq -\int_0^{t}du \int_{\substack{\lvert \xt{p} \rvert< 1/(2\lvert\xt{q}\rvert)\\ \lvert \xt{p} \rvert< 1/(2\mu t\lvert\xt{y}\rvert) } } d\xt{p} \sigma(\xt{p})(1-\cos(\xt{p}\cdot(\xt{q}+u\mu \xt{y})))\\
&\leq -\int_0^{t}du \int_{\substack{\lvert \xt{p} \rvert< 1/(2\lvert\xt{q}\rvert)\\ \lvert \xt{p} \rvert< 1/(2\mu t\lvert\xt{y}\rvert) } } d\xt{p} \sigma(\xt{p})\lvert \xt{p}\cdot(\xt{q}+u\mu \xt{y})\rvert^2
\end{split}\]
Then,
\[\begin{split}
\int_0^{t} du \lvert \xt{p}\cdot(\xt{q}+u\mu \xt{y})\rvert^2&= t\lvert \xt{p}\cdot\xt{q}\rvert ^2+ \mu t^2 \xt{p}\cdot\xt{q}\,\, \xt{p}\cdot\xt{y}+\mu^2 \frac{t^3}{3}\lvert \xt{p}\cdot\xt{y}\rvert ^2\\
&= \big( t-\mu\frac{t^2}{2}\big)\lvert \xt{p}\cdot\xt{q}\rvert ^2 + \big(\mu^2\frac{t^3}{3}-\mu\frac{t^2}{2}\big)\lvert \xt{p}\cdot\xt{y}\rvert ^2+\mu\frac{t^2}{2}\lvert \xt{p}\cdot (\xt{q}+\xt{y})\rvert^2\\
&\geq  \underbrace{t\big( 1-\mu\frac{t}{2}\big)}_{\geq 0}\lvert \xt{p}\cdot\xt{q}\rvert ^2 + \underbrace{\mu t^2\big(\mu\frac{t}{3}-\frac{1}{2}\big)}_{\geq 0}\lvert \xt{p}\cdot\xt{y}\rvert ^2,
\end{split}\]
since  $t\in(3/(2\mu),2/\mu)$. Consequently,
\[
\int_0^{t}du \Psi(\xt{q}+u\mu\xt{y}) \leq  -C_{t,\mu} (\lvert \xt{p}\cdot\xt{q}\rvert ^2 + \lvert \xt{p}\cdot\xt{y}\rvert ^2 ),
\]
and
\[\begin{split}
\exp\Big(\int_0^{t}du \Psi(\xt{q}+u\mu \xt{y})\Big) &\leq \exp(-C_{t,\mu}    \int_{\substack{\lvert \xt{p} \rvert< 1/(2\lvert\xt{q}\rvert)\\ \lvert \xt{p} \rvert\leq 1/(2\mu t\lvert\xt{y}\rvert) } } d\xt{p} \sigma(\xt{p}) \lvert \xt{p}\cdot\xt{q}\rvert ^2 )\1_{(\lvert \xt{q}\rvert \geq \mu t \lvert \xt{y}\vert)}\\
&+\exp(- C_{t,\mu}\int_{\substack{\lvert \xt{p} \rvert< 1/(2\lvert\xt{q}\rvert)\\ \lvert \xt{p} \rvert< 1/(2\mu t\lvert\xt{y}\rvert) } } d\xt{p}\sigma(\xt{p}) \lvert \xt{p}\cdot\xt{y}\rvert ^2)\1_{(\lvert \xt{q}\rvert< \mu t \lvert \xt{y}\vert)}.
\end{split}\]

Moreover, according to \eqref{SDCAreg}, there exists $M_{t,\mu}$ large enough so that $\forall \lvert \xt{q}\rvert \geq M_{t,\mu}$ and $\forall \lvert \xt{y}\rvert \geq M_{t,\mu}$,

\begin{equation}\label{expterm}\begin{split}
\exp\Big(\int_0^{t}du \Psi(\xt{q}+u\mu \xt{y})\Big)&\leq \exp\Big(- C_{t,\mu}\lvert \xt{q}\rvert^\theta \int_{\lvert \xt{p} \rvert\leq 1/2  } d\xt{p} \frac{1}{\lvert \xt{p}\rvert^{d+\theta}}\big\lvert \xt{p}\cdot\frac{\xt{q}}{\lvert \xt{q}\rvert}\big\rvert ^2\Big)\1_{(\lvert \xt{q}\rvert\geq \mu t \lvert \xt{y}\vert)}\\
&+\exp\Big(- C_{t,\mu}\lvert \xt{y}\rvert^\theta \int_{\lvert \xt{p} \rvert\leq 1/2  } d\xt{p} \frac{1}{\lvert \xt{p}\rvert^{d+\theta}} \big\lvert \xt{p}\cdot\frac{\xt{y}}{\lvert \xt{y}\rvert}\big\rvert ^2\Big)\1_{(\lvert \xt{q}\rvert< \mu t \lvert \xt{y}\vert)}.
\end{split}\end{equation}
Finally, thanks to a unitary transformation on $\mathbb{S}^{d-1}$,
\[ \int_{\lvert \xt{p} \rvert\leq 1/2  } d\xt{p} \frac{1}{\lvert \xt{p}\rvert^{d+\theta}} \big\lvert \xt{p}\cdot\frac{\xt{z}}{\lvert \xt{z}\rvert}\big\rvert ^2=C_\theta \int_{\mathbb{S}^{d-1}} \lvert \xt{u}\cdot\xt{v}\rvert^2 dS(\xt{u}),\]
 where $\xt{v}\in \mathbb{S}^{d-1}$ is fixed.
$\square$
\end{proof}

The following lemma shows the transport of regularity from the variable $\xt{k}$ to the variable $\xt{x}$  through the transport term in \eqref{radtranseq}. This transport of regularity is a direct consequence of Lemma \ref{lemkey}.

\begin{lem}\label{lemreg2}
Let $W$ be the unique weak solution uniformly bounded in $ L^2(\mathbb{R}^{2d})$ of \eqref{radtranseq}. We have $\forall t_1> 0$
\[W\in L^\infty \Big([t_1,+\infty),\bigcap _{k\geq 0} H^k (\mathbb{R}^{2d})\Big) \] 
\end{lem}
\begin{proof}[of Lemma \ref{lemreg2}]
Let $t_1>0$, $t>t_1$, and $\eta>0$, such that $t\in(3/(2\mu),2/\mu)$. Then, using a change of variable, we have
\[\begin{split}
\|W(t)\|^2_{H^k(\mathbb{R}^{2d})}\leq& C_\mu \int_{\substack{\lvert \xt{q}\rvert \geq M_{t,\mu}\\ \lvert \xt{y}\rvert \geq M_{t,\mu}}} d\xt{y}d\xt{q} (\lvert \xt{q}\rvert^k+\lvert \xt{y}\rvert^k  ) e^{\int_0^{t}du \Psi(\xt{q}+u\mu \xt{y})}\lvert \widehat{W}_0(\xt{y},\xt{q}+\mu t \xt{y})\rvert ^2\\
&+ C_\mu  \int_{\substack{\lvert \xt{q}\rvert \geq M_{t,\mu}\\ \lvert \xt{y}\rvert \leq M_{t,\mu}}} d\xt{y}d\xt{q} \lvert \xt{q}\rvert^k e^{\int_0^{t}du \Psi(\xt{q}+u\mu \xt{y})}\lvert \widehat{W}_0(\xt{y},\xt{q}+\mu t \xt{y})\rvert ^2\\
&+ C_\mu  \int_{\substack{\lvert \xt{q}\rvert \leq M_{t,\mu}\\ \lvert \xt{y}\rvert \geq M_{t,\mu}}} d\xt{y}d\xt{q} \lvert \xt{y}\rvert^k e^{\int_0^{t}du \Psi(\xt{q}+u\mu \xt{y})}\lvert \widehat{W}_0(\xt{y},\xt{q}+\mu t \xt{y})\rvert ^2\\
&+ C_\mu  \int_{\substack{\lvert \xt{q}\rvert \leq M_{t,\mu}\\ \lvert \xt{y}\rvert \leq M_{t,\mu}}} d\xt{y}d\xt{q}\lvert \widehat{W}_0(\xt{y},\xt{q}+\mu t \xt{y})\rvert ^2.
\end{split}\]
The first and last terms are finite since $W_0\in L^2(\mathbb{R}^{2d})$ and thanks to Lemma \ref{lemkey}.  The second and third terms are finite using \eqref{expterm} depending on whether $\xt{q}$ or $\xt{y}$ is bounded.
 $\square$
\end{proof}

Finally, from Lemma \ref{lemreg2} we have the continuity of $W$.

\begin{lem}\label{lemreg3}
Let $W$ be the unique weak solution uniformly bounded in $ L^2(\mathbb{R}^{2d})$ of \eqref{radtranseq}. We have
\[W\in \mathcal{C}^0 \Big((0,+\infty),\bigcap _{k\geq 0} H^k (\mathbb{R}^{2d})\Big). \] 
\end{lem}
\begin{proof}[of Lemma \ref{lemreg3}]
Let $t_1>0$ and $(t_n)_n$ be a sequence which converges to $t_1$, Then, there exists $N$ and $t_0>0$ such that $\forall n\geq N$, $t_1\wedge t_n>t_0$, and 
\[\begin{split} \|W(t_n)-W(t_1)\|^2_{H^k(\mathbb{R}^{2d})}\leq &C_k \int d\xt{y}d\xt{q}( \lvert \xt{y} \rvert^k+\lvert \xt{q} \rvert^k) \lvert  \widehat{W}(t_0,\xt{y},\xt{q}+t_1\xt{y})\rvert^2\\
&\times\big\lvert 1-e^{\int_{t_n}^{t_1}\Psi(\xt{q}+(v-t_0)\xt{y})}\big\rvert ^2\\
&+C_k \int d\xt{y}d\xt{q}( \lvert \xt{y} \rvert^k+\lvert \xt{q} \rvert^k) \lvert  \widehat{W}(t_0,\xt{y},\xt{q}+t_n\xt{y})- \widehat{W}(t_0,\xt{y},\xt{q}+t_1\xt{y})\rvert^2.
\end{split}\] 
Then, we can conclude from Lemma \ref{lemreg2} and the dominated convergence theorem.
$\square$
\end{proof}

\section{Proof of Proposition \ref{probrep}}\label{proofproprep}

First, let $W_0\in \mathcal{C}^\infty _0(\mathbb{R}^{2d})$ and let $(L_t)_t$ be a Lévy process with characteristic exponent $\Psi$ defined by \eqref{caracexp} and infinitesimal generator given by $\mathcal{L}$ defined by \eqref{formgene}, and so that $(L_t)_t$ is a Markov process on $\mathcal{C} _{0,\infty} (\mathbb{R}^{2d})$, the space of continuous functions which tend to $0$ at infinity. Moreover, let us note that $\mathcal{C}^\infty _0 (\mathbb{R}^{2d})\subset \mathcal{D}(\mathcal{L})$.

According to assumption \eqref{SDCAreg}, the fact that $\sigma$ has a support included in a compact set, and \cite[Lemma 31.7]{sato}, the process defined by
\begin{equation}\label{markovcouple} X_t=\Big(\xt{x}-t\xt{k}-\int_0^t L_s ds,\xt{k}+L_t\Big)\end{equation}
is a Markov process on $\mathcal{C}_{0,\infty}(\mathbb{R}^{2d})$ with infinitesimal generator given by
\[L\varphi (\xt{x},\xt{k})=-\xt{k}\cdot\nabla_{\xt{x}}\varphi(\xt{x},\xt{k})+\mathcal{L}\varphi(\xt{x},\xt{k}).\] 
As the result, The function defined by
\begin{equation}\label{repW}W(t,\xt{x},\xt{k})=\E\Big[W_0\Big(\xt{x}-t\xt{k}-\int_0^t L_s ds,\xt{k}+L_t\Big)\Big],\end{equation}
is a classical solution of 
\[\partial_t W+\xt{k}\cdot\nabla W=\mathcal{L}W.\]
Moreover, if $p(t,d\xt{x}d\xt{k})$ is the transition function of the Markov process $(\int_0^t L_s ds, L_t)_t$, 
\begin{equation}\label{boundext}\begin{split}
\| W(t) \|^2_{L^2(\mathbb{R}^{2d})}&=\int d\xt{x}d\xt{k}\Big\lvert \int W_0(\xt{x}-t\xt{k}-\xt{q},\xt{k}+\xt{q})  p(t,d\xt{y}d\xt{q})\Big\rvert^2 \\
&\leq \int  p(t,d\xt{y}d\xt{q})  \int d\xt{k} \int d\xt{x} \rvert W_0(\xt{x}-t\xt{k}-\xt{q},\xt{k}+\xt{q})\lvert^2\\
&\leq \|W_0\|^2_{L^2(\mathbb{R}^{2d})},
\end{split}\end{equation}
so that $W$ is uniformly bounded in $L^2(\mathbb{R}^{2d})$ and is the unique weak solution of \eqref{radtranseq}. However, \eqref{boundext} shows that \eqref{repW} can be extend to $W_0\in L^2(\mathbb{R}^{2d})$.
Then by density,  \eqref{repW}   is the unique weak solution uniformly bounded of \eqref{radtranseq} with $W_0\in L^2(\mathbb{R}^{2d})$. That concludes the proof of Proposition \ref{probrep}.

\section{Proof of Theorem \ref{fractional}}\label{proofthfrac}

First, Let us recall that 
\begin{equation}\label{unifboundW}
\sup_\eta \sup_{t\geq 0} \|W^\eta(t)\|_{L^2(\mathbb{R}^{2d})}\leq \|W_0\|_{L^2(\mathbb{R}^{2d})},\end{equation}
where $W^{\eta}$ is the unique solution uniformly bounded in $L^2(\mathbb{R}^{2d})$ of the radiative transfer equation \eqref{radtranseq} associated to $\sigma^\eta(\xt{p})$.

The proof of Theorem \ref{fractional} uses the probabilistic representation obtained in Proposition \ref{probrep}.  The following proof is in two step. First, we prove the relative compactness of $(W^{\eta})_\eta$, afterward we identify all the subsequence limits. The proof of the relative compactness of $(W^{\eta})_\eta$ is based on the probabilistic representation obtained in Proposition \ref{probrep}, but where $(L^\eta_t)_t$ is associated to the jump measure $\sigma^\eta(\xt{p})d\xt{p}$.

\begin{lem}\label{relatcomp}
$\forall (\xt{x},\xt{k})\in\mathbb{R}^{2d}$, $(X^\eta)_\eta$, defined by \eqref{markovcouple}, is tight is on $\mathcal{D}([0,+\infty),\mathbb{R}^{2d})$. 
\end{lem}

\begin{proof}[of Lemma \ref{relatcomp}]
To prove this lemma we will use the Aldous tightness criteria \cite{billingsley}. Let us begin with the tightness of $(L^\eta_t)_t$. 
First of all, let us decompose $(L^\eta_t)_t$ according to the size of its jumps,
\[L^\eta_t =L^{1,\eta}_t+L^{2,\eta}_t =  \int_{\lvert \xt{p}\rvert \leq 1} \xt{p}N^\eta (t,d\xt{p})+\int_{\lvert \xt{p}\rvert > 1} \xt{p}N^\eta (t,d\xt{p}),\]
where $N^\eta(t,d\xt{p})$ is a random Poisson measure with intensity measure $\sigma^\eta(\xt{p})d\xt{p}$. Let us note that $(L^{1,\eta}_t)_t$ is a Lévy process with compactly supported jump measure $\sigma^\eta(\xt{p})\1_{\lvert \xt{p}\rvert \leq 1}d\xt{p}$, so that 
\[\sup_{\eta}\sup_{0\leq t \leq T }\E[\lvert L^{1,\eta}_t \rvert ^2]<+\infty,\]
thanks to \eqref{boundedfrac} and \cite[Theorem 2.3.7]{applebaum}. Moreover, $(L^{2,\eta}_t)_t$ is a compound Poisson process \cite[Theorem 2.3.9]{applebaum}. Then, there exist a poisson process $(N_t)_t$ with intensity $1$ and a sequence $(Z^\eta _n)_n$ of iid random variables with distribution $\sigma^\eta(\xt{p})\1_{\lvert \xt{p}\rvert > 1}d\xt{p}$, such that 
\[L^{2,\eta}_t=\sum_{j=0}^{N(t)} Z^\eta _j\]
 in distribution with $Z^\eta_0=0$. Moreover, $(N_t)_t$ and $(Z^\eta _n)_n$ are independent. 
First,
\[\mathbb{P}(\sup_{0\leq t\leq T} \lvert L^\eta_t \rvert \geq M)\leq  \mathbb{P}(\sup_{0\leq t\leq T} \lvert L^{1,\eta}_t \rvert \geq M/2)+\mathbb{P}(\sup_{0\leq t\leq T} \lvert L^{2,\eta}_t \rvert \geq M/2).\]
Since 
\[ L^{1,\eta}_t-t\int_{\lvert \xt{p} \rvert \leq 1} \xt{p} \sigma^\eta(\xt{p}) d\xt{p}\]
is a martingale, and for $M$ sufficiently large, 
\[\begin{split}
 \mathbb{P}(\sup_{0\leq t\leq T} \lvert L^{1,\eta}_t \rvert \geq M/2)&\leq  \mathbb{P}\Big(\sup_{0\leq t\leq T} \Big\lvert L^{1,\eta}_t -t\int_{\lvert \xt{p} \rvert \leq 1} \xt{p} \sigma^\eta(\xt{p}) d\xt{p}\Big\rvert \geq M/4\Big)\\
 &\leq \frac{32}{M^2}\Big[\sup_{\eta}\E[\lvert L^{1,\eta}_t \rvert ^2] + CT\Big( \int_{\lvert \xt{p} \rvert \leq 1} \frac{1}{\lvert \xt{p}\rvert ^{d+\theta-1}} d\xt{p} \Big)^2\Big],
\end{split}\]
according to the Doob's maximal inequality. Moreover, 
\[\begin{split}
\mathbb{P}(\sup_{0\leq t\leq T} \lvert L^{2,\eta}_t \rvert \geq M/2)&\leq \sum_{n\geq0} \mathbb{P}\Big(\sum_{j=0}^n \lvert Z^\eta_j \rvert \geq M/2\Big\vert N(T)=n \Big)\mathbb{P}(N(T)=n)\\
&\leq \sum_{n\geq 0} \sum_{j=1}^n \mathbb{P}(\lvert Z^\eta_j\rvert \geq M/(2n))\mathbb{P}(N(T)=n)\\
&\leq  \sum_{n\geq 0} n \int_{\lvert\xt{p} \rvert \geq M/(2n) } \sigma^\eta(\xt{p}) d\xt{p} \mathbb{P}(N(T)=n)\\
&\leq \frac{C \Omega (d)}{M^{\theta}}\E[N(T)^{1+\theta}].
\end{split}\]
Then,
\[\sup_{\eta}\mathbb{P}(\sup_{0\leq t\leq T} \lvert L^\eta_t \rvert \geq M)=\mathcal{O}\Big(\frac{1}{M^{\theta}}\Big)\quad \text{as }M\to +\infty.\]
Consequently, $\forall M'>0$
\[ \mathbb{P}\Big(\sup_{0\leq t\leq T}\Big\lvert\int_0^t L^\eta _s ds \Big\rvert \geq M \Big)\leq \mathbb{P}\Big(\sup_{0\leq t\leq T}\lvert L^\eta _t\rvert \geq M'\Big)+\mathbb{P}\Big(\int_0^T \lvert L^\eta _s\rvert ds \geq M ,\sup_{0\leq t\leq T}\lvert L^\eta _t\rvert \leq M' \Big), \]
and then,
\[\lim_{M\to +\infty}\sup_{\eta}\mathbb{P}(\sup_{0\leq t\leq T} \lvert X^\eta(t)\rvert \geq M  )=0.\]
Let $\tau$ be a discrete stopping time relatively to the filtration of $(X^\eta)_\eta$, and bounded by $T$. According to the stationarity of $(L^\eta _t)_t$,
$\forall \mu >0$ and $\forall \nu >0$, we have
\[\begin{split}
\mathbb{P}(\lvert X^{\eta}(\tau +\mu)-X^{\eta}(\tau)\rvert >\nu)&\leq \mathbb{P}(\lvert L^{\eta}(\tau +\mu)-L^{\eta}(\tau)\rvert >\nu/2)+\mathbb{P}\Big(\int^{\tau +\mu}_{\tau}\lvert L^{\eta}_s\rvert ds >\nu/2\Big) \\
&\leq  \mathbb{P}(\lvert L^{\eta}(\mu)\rvert >\nu/2)+\mathbb{P}\Big(\int^\mu_0\lvert L^{\eta}_s\rvert ds >\nu/2\Big).\end{split}\]
First, it is clear that $\forall M'>0$
\[ \mathbb{P}\Big(\int^\mu_0\lvert L^{\eta}_s\rvert ds >\nu/2\Big)\leq  \mathbb{P}\Big(\sup_{0\leq t\leq T}\lvert L^\eta _t\rvert \geq M'\Big)+\mathbb{P}\Big(\int_0^\mu \lvert L^\eta _s\rvert ds \geq \nu/2 ,\sup_{0\leq t\leq T}\lvert L^\eta _t\rvert \leq M' \Big).\]
Moreover,
\[\begin{split}
 \mathbb{P}(\lvert L^{\eta}(\mu)\rvert >\nu/2)&\leq  \mathbb{P}(\lvert L^{1,\eta}(\mu)\rvert >\nu/4)+ \mathbb{P}(\lvert L^{2,\eta}(\mu)\rvert >\nu/4)\\
&\leq \frac{16}{\nu^2} \Big(\mu \int_{\lvert \xt{p}\rvert \leq 1}d\xt{p}\frac{C}{\lvert \xt{p}\rvert^{d+\theta-1}}+\mu^2 \Big(\int_{\lvert \xt{p}\rvert \leq 1}d\xt{p}\frac{C}{\lvert \xt{p}\rvert^{d+\theta-1}}\Big)^2\Big)+C_{d,\theta}\E[N(\mu)^{1+\theta}].
\end{split}\]
However, $\E[N(\mu)^{1+\theta}]\leq \E[N(\mu)^2]=\mu+\mu^2$, so that 
\[ \lim_{\mu\to 0}\sup_{\eta}\sup_{\tau} \mathbb{P}(\lvert X^{\eta}(\tau +\mu)-X^{\eta}(\tau)\rvert >\nu)=0.\]
That concludes the proof of Lemma \ref{relatcomp}.
$\square$
\end{proof}

Lemma \ref{relatcomp} implies the relative compactness of $(W^{\eta})_\eta$. Now, let us identify all the accumulation points of $(W^{\eta})_\eta$. First, let us denote 
\[  \mathcal{L}^{\infty}\varphi(\xt{k})=\int d\xt{p}\sigma^{\infty}(\xt{p})\big(\varphi(\xt{k}+\xt{p})-\varphi(\xt{k})\big)\quad \text{and} \quad \sigma^{\infty}(\xt{p})=\frac{\sigma}{(2\pi)^d\lvert \xt{p}\rvert^{d+\theta}}.  \]
We have the following lemma which shows the convergence of the transfer operator $\mathcal{L}^\eta$. 
\begin{lem}\label{cvgeneW}
We have $\forall \lambda\in \mathcal{C}^\infty_0(\mathbb{R}^{d})$, 
\[\lim_{\eta \to 0}\|\mathcal{L}^{\eta}\lambda-\mathcal{L}^\infty \lambda\|_{L^2(\mathbb{R}^{d})}=0.\]
\end{lem}
\begin{proof}[of Lemma \ref{cvgeneW}]
This lemma is a consequence of
\[\begin{split}
\| \mathcal{L}^{\eta}\lambda-\mathcal{L}^\infty \lambda\|_{L^2(\mathbb{R}^{2d})}\leq C(& \| \lambda \|_{L^2(\mathbb{R}^{2d})}+\| \nabla_{\xt{k}}\lambda \| )\\
&\times\Big(\int_{\lvert \xt{p}\rvert \leq 1}d\xt{p} \lvert \xt{p}\rvert \lvert \sigma^\eta(\xt{p})-\sigma^\infty(\xt{p})\rvert +\int_{\lvert \xt{p}\rvert > 1}d\xt{p}  \lvert \sigma^\eta(\xt{p})-\sigma^\infty(\xt{p})\rvert\Big), \end{split}\]
and the dominated convergence theorem, because of \eqref{globalfrac} and \eqref{boundedfrac}.
$\square$
\end{proof}
Now, let us consider $\eta_n$ a sequence, which goes to $0$ as $n\to +\infty$, and such that $\forall (t,\xt{x},\xt{k})$
\[\lim_{n\to+\infty}W^{\eta_n}(t,\xt{x},\xt{k})=W^{\infty}(t,\xt{x},\xt{k}).\] 
\begin{lem}\label{caractW}
$W^{\infty}(t,\xt{x},\xt{k})$ is a weak solution uniformly bounded in $L^2(\mathbb{R}^{2d})$ of the radiative transfer equation
\begin{equation}\label{radtreqW}
\partial_tW^\infty+\xt{k}\cdot\nabla_{\xt{x}}W^\infty=\mathcal{L}^\infty W^\infty\quad \text{with}\quad W^\infty(0,\xt{x},\xt{k})=W_0(\xt{x},\xt{k}).
\end{equation} 

\end{lem}
\begin{proof}[of Lemma \ref{caractW}]
First, by Fatou's Lemma, we have $\forall t\geq0$
\[\|W^{\infty}(t)\|_{L^2(\mathbb{R}^{2d})}\leq \liminf_{n}  \|W^{\eta_n}(t)\|_{L^2(\mathbb{R}^{2d})}\leq \|W_0\|_{L^2(\mathbb{R}^{2d})}.\]
From this result, using a molifier, and the dominated convergence theorem, we have that 
\[\forall t\geq 0, \quad \lim_{n\to+\infty}W^{\eta_n}(t)=W^{\infty}(t)\text{ weakly in }L^2(\mathbb{R}^{2d}). \]
According to \eqref{unifboundW}, Lemma \ref{cvgeneW}, the dominated convergence theorem, and taking the limit in
\[\big<W^{\eta_n}(t),\lambda\big>_{L^2(\mathbb{R}^{2d})}-\big<W_0,\lambda\big>_{L^2(\mathbb{R}^{2d})}-\int_0^t\big<W^{\eta_n}(s),\xt{k}\cdot\nabla_{\xt{x}}\lambda+\mathcal{L}^{\eta_n}\lambda\big>_{L^2(\mathbb{R}^{2d})}ds=0\]
$\forall \lambda\in \mathcal{C}^\infty _0(\mathbb{R}^{2d})$, we get the desired result.
$\square$
\end{proof}
The following lemma concludes the proof of Theorem \ref{fractional}.
\begin{lem}\label{uniqW}
Equation \eqref{radtreqW} admits a unique weak solution uniformly bounded in $L^2(\mathbb{R}^{2d})$, which is given by \eqref{formula}. 
\end{lem}
\begin{proof}[of Lemma \ref{uniqW}]
First of all, let us recall the following Lévy-Khintchine formula \cite[Thorem 14.14]{sato}
\[\int \frac{d\xt{p}}{\lvert \xt{p}\rvert^{d+\theta} }\big(e^{i\xt{p}\cdot\xt{q}}-1\big)=-\lvert \xt{q}\rvert ^\theta\theta\Gamma(1-\theta)\int_{\mathbb{S}^{d-1}}dS(\xt{u})\lvert \xt{e}_1\cdot \xt{u} \rvert ^\theta,\]
which give $\forall \lambda\in \mathcal{C}^\infty_0(\mathbb{R}^{2d})$
\[\mathcal{L}^\infty\lambda=-\sigma(\theta)(-\Delta)^{\theta/2}\lambda.\]
To show the uniqueness of Equation \eqref{radtreqW}, let us defined

\[\lambda(t,\xt{x},\xt{k})=\frac{1}{(2\pi)^d}\int d\xt{y}d\xt{q} e^{i(\xt{x}\cdot\xt{y}+\xt{k}\cdot\xt{q})}\exp\Big(-\sigma(\theta)\int_0 ^{T-t}\lvert \xt{q}+u\xt{y} \rvert^\theta\Big)\widehat{\lambda}_0(\xt{y},\xt{q}+(T-t)\xt{y}),\]

with $\lambda_0\in \mathcal{C}^\infty _0(\mathbb{R}^{2d})$, and let $W$ be a solution of \eqref{radtreqW} uniformly bounded in $L^2(\mathbb{R}^{2d})$ with $W(0)=0$. Then, since $\mathcal{D}(\mathcal{L}^\infty)=H^\theta(\mathbb{R}^d)$, we have
\[\big<W(T),\lambda_0\big>_{L^2(\mathbb{R}^{2d})}=\big<W(T),\lambda(T)\big>_{L^2(\mathbb{R}^{2d})}=\int_0^T ds \big<W(s),\partial_t\lambda+\xt{k}\cdot\nabla_{\xt{x}}\lambda+\mathcal{L}^\infty \lambda\big>_{L^2(\mathbb{R}^{2d})}=0.\]
Moreover, it is clear that \eqref{formula} is a weak solution for $W_0\in \mathcal{C}^\infty_0(\mathbb{R}^{2d})$. Consequently, by density of $\mathcal{C}^\infty_0(\mathbb{R}^{2d})$ in $L^2(\mathbb{R}^{2d})$,  \eqref{formula} is also a weak solution uniformly bounded in $L^2(\mathbb{R}^{2d})$ with $W_0\in L^2(\mathbb{R}^{2d})$. 
$\square$
\end{proof}

\section*{Conclusion}

In this paper we have analyzed the asymptotic phase space energy density of a solution of the random Schrödinger equation with long-range correlation properties. The phase of a solution of the random Schrödinger equation in the same context has been studied in \cite{bal2}, and the phase and the phase space energy density for the random Schrödinger equation with rapidly decaying correlations has been studied in \cite{bal0,bal6,bal2}. In the case of rapidly decaying correlations the phase and the phase space energy density evolve on the same scale, and the asymptotic evolution of the phase space energy density is given by a radiative transfer equation. In the other hand, for random perturbations with slowly decaying correlations, the phase and the phase space energy density do not evolve on the same scale anymore. However, the scale of evolution of the phase space energy density in the context of rapidly and slowly decaying correlations are the same, and their evolutions are given by radiative transfer equations of the same form, but nevertheless with thorough differences. In the context of perturbations with long-range correlations, a scattering coefficient cannot be defined because of a nonintegrable singularity in the radiative transfer equation. However, this singularity is responsible to a regularizing effect on this radiative transfer equation which cannot be observed in the case of rapidly decaying correlations. Finally, we have derived an approximation of the radiative transfer equation involving a fractional Laplacian, which gives us a simpler model of phase space energy distribution which corresponds to the long space and time diffusion approximation of the radiative transfer equation. Moreover, this model permits to exhibit a decay of the phase space energy distribution which obey to a power law with exponent lying in $(0,1)$.

\end{document}